\RequirePackage{plautopatch}
\documentclass{amsart}

\usepackage{amsmath,amssymb,amsthm,mathtools}
\usepackage{enumitem}
\usepackage{iftex}
\ifptex
  \usepackage[dvipdfmx,hidelinks]{hyperref}
\else
  \usepackage[hidelinks]{hyperref}
\fi
\usepackage{cleveref}

\numberwithin{equation}{section}

\theoremstyle{plain}
\newtheorem{theorem}{Theorem}[section]
\newtheorem{proposition}[theorem]{Proposition}
\newtheorem{lemma}[theorem]{Lemma}
\newtheorem{corollary}[theorem]{Corollary}
\newtheorem{introtheorem}{Theorem}

\crefname{theorem}{theorem}{theorems}
\Crefname{theorem}{Theorem}{Theorems}
\crefname{proposition}{proposition}{propositions}
\Crefname{proposition}{Proposition}{Propositions}
\crefname{lemma}{lemma}{lemmas}
\Crefname{lemma}{Lemma}{Lemmas}
\crefname{corollary}{corollary}{corollaries}
\Crefname{corollary}{Corollary}{Corollaries}
\crefname{introtheorem}{theorem}{theorems}
\Crefname{introtheorem}{Theorem}{Theorems}

\theoremstyle{definition}
\newtheorem{definition}[theorem]{Definition}
\crefname{definition}{definition}{definitions}
\Crefname{definition}{Definition}{Definitions}

\theoremstyle{remark}

\crefname{remark}{remark}{remarks}
\Crefname{remark}{Remark}{Remarks}
\newtheorem*{ack}{Acknowledgment}

\crefname{section}{section}{sections}
\Crefname{section}{Section}{Sections}
\crefname{equation}{equation}{equations}
\Crefname{equation}{Equation}{Equations}

\setlist[enumerate]{leftmargin=2em,itemsep=0.25em,topsep=0.25em}
\newcommand{\R}{\mathbb{R}}
\newcommand{\C}{\mathbb{C}}
\newcommand{\N}{\mathbb{N}}
\newcommand{\Z}{\mathbb{Z}}
\newcommand{\Lip}{\operatorname{Lip}}
\newcommand{\Lipb}{\operatorname{Lip}_{\mathrm b}}
\newcommand{\lipa}{\operatorname{lip}_{\mathrm a}}
\newcommand{\Var}{\operatorname{Var}}
\newcommand{\Sep}{\operatorname{Sep}}
\newcommand{\supp}{\operatorname{supp}}
\newcommand{\tr}{\operatorname{tr}}

\newcommand{\ThetaPI}{\Theta_{\mathrm{PI}}}
\newcommand{\dd}{\,d}
\newcommand{\ind}{\mathbf{1}}
\newcommand{\norm}[1]{\lVert #1\rVert}
\newcommand{\abs}[1]{\lvert #1\rvert}
\newcommand{\inner}[2]{\langle #1,#2\rangle}

\title[Hopf quotients of Gaussian pyramid]{Hopf quotients of the infinite-dimensional Gaussian pyramid}
\author{Takashi Shioya}
\author{Shigeaki Yokota}
\address{Mathematical Institute, Tohoku University, Sendai 980-8578, Japan}
\email{shioya@math.tohoku.ac.jp}
\email{shigeaki.yokota.t4@dc.tohoku.ac.jp}
\date{}
\thanks{This work was supported by JSPS KAKENHI Grant Number 24K06729.}

\begin{document}

\begin{abstract}
We study the infinite-dimensional Gaussian pyramid and its quotients by the global sign flip and the $U(1)$-Hopf action. We resolve affirmatively a long-standing problem posed by Tomohiro Fukaya around 2014: these three limiting geometries are pairwise non-similar, meaning that no positive rescaling makes any two of them coincide.
\end{abstract}

\maketitle

\section*{Introduction}

A basic question in high-dimensional metric measure geometry is which geometric distinctions survive as the dimension tends to infinity.
Concentration of measure provides a framework for studying such limits through Lipschitz observables, even when the limiting object cannot be represented by an ordinary metric measure space.
In this paper, we ask whether the distinction between spherical, real projective, and complex projective geometry persists at the level of their pyramidal limits, even after arbitrary positive rescaling.

Gromov \cite[\S $3\frac12$]{gromov2007met} developed a convergence theory based on concentration of measure for \emph{metric measure spaces}, or \emph{mm-spaces}: complete separable metric spaces equipped with Borel probability measures.
To accommodate more general limits, he introduced the notion of a \emph{pyramid}, a nonempty, box-closed family of isomorphism classes of mm-spaces that is downward closed under the Lipschitz order and upward directed with respect to that order.
The space of pyramids, equipped with the weak topology, forms the pyramidal compactification of the space of mm-spaces with the concentration topology.
We refer to \cite{ShioyaBook} and \Cref{sec:setup} for details.

The basic examples arise from high-dimensional spheres and real and complex projective spaces.
Shioya \cite{ShioyaBook,ShioyaCP} proved that the spheres $S^n(\sqrt n)\subset\R^{n+1}$ converge weakly to the Gaussian pyramid $\Gamma^\infty$, also called the infinite-dimensional virtual standard Gaussian space.
This pyramid is the box closure of the union of the pyramids associated with finite-dimensional standard Gaussian spaces.
He also proved that the real and complex projective spaces obtained as the quotients $S^n(\sqrt n)/\Z_2$ and $S^{2n+1}(\sqrt{2n+1})/U(1)$ converge weakly to $\Gamma^\infty/\Z_2$ and $\Gamma^\infty/U(1)$, respectively.
These results use the Euclidean metric restricted to the spheres and the induced quotient metrics.
Kazukawa \cite{kazukawa} subsequently established the corresponding results for the canonical Riemannian metrics.

Further geometric examples and analytic developments have extended this theory.
Shioya and Takatsu \cite{ShioyaTakatsu} obtained Gaussian spaces and their quotients as high-dimensional limits of Stiefel and flag manifolds, and Kazukawa and Shioya \cite{KazukawaShioyaEllipsoids} studied Gaussian limits of ellipsoids.
Esaki, Kazukawa and Mitsuishi \cite{EKM} developed a theory of pyramid invariants for distinguishing limits that are not represented by ordinary mm-spaces.
In a related analytic direction, Gigli and Vincini \cite{GigliVincini} established stability results for Sobolev energies and heat flow under concentration convergence, including heat-flow convergence under a uniform curvature-dimension bound.

More recently, Gigli, Suzuki and Vincini \cite{GigliSuzukiVincini} represented pyramids by extended metric measure spaces and studied the stability of functional inequalities under weak convergence of pyramids.
In particular, they identified the Gaussian pyramid with the box closure of the pyramid associated with an abstract Wiener space equipped with the Cameron--Martin extended distance.
This gives a concrete connection between pyramidal limits and infinite-dimensional Gaussian analysis.

For the basic spherical and projective models, the convergence theorems above describe the limits, but do not by themselves distinguish them.
The remaining question is whether the sign-flip and $U(1)$-Hopf actions leave different metric measure structures in the limit, or whether their effects can be absorbed into a change of scale.
Around 2014, Tomohiro Fukaya raised the following question, which has subsequently been discussed in the literature \cite{EKM,ShioyaProb}.

\medskip\noindent
\textbf{Problem.}
Are the three pyramids $\Gamma^\infty$, $\Gamma^\infty/\Z_2$ and $\Gamma^\infty/U(1)$ pairwise non-similar?

\medskip
We say that two pyramids $\mathcal P$ and $\mathcal Q$ are \emph{similar} if $a\mathcal P=\mathcal Q$ for some $a>0$, where $a\mathcal P$ is obtained by multiplying the metrics of all its members by $a$.
In this paper we resolve this long-standing problem affirmatively.

\begin{introtheorem} \label{thm:main}
The three pyramids $\Gamma^\infty$, $\Gamma^\infty/\Z_2$ and $\Gamma^\infty/U(1)$ are pairwise non-similar.
\end{introtheorem}

In general, determining whether two pyramids are similar is difficult. Esaki, Kazukawa and Mitsuishi \cite{EKM} proved that the infinite product $[0,1]^\infty$ of the interval $[0,1]$ is not similar to $\Gamma^\infty$ by using the $(2,2)$-Poincar\'e constant and introducing a new invariant of pyramids. They extended the $(2,2)$-Poincar\'e constant of an mm-space to that of a pyramid.
We denote by $C_{2,2}(\mathcal{P})$ the $(2,2)$-Poincar\'e constant of a pyramid $\mathcal{P}$.

The strategy of our proof is as follows.
We first consider the $(2,2)$-Poincar\'e constants.
Esaki, Kazukawa and Mitsuishi \cite{EKM} already proved
\[
C_{2,2}(\Gamma^\infty) = 1.
\]
We determine the constants of the two quotient pyramids as follows.
\begin{introtheorem} \label{thm:C22}
\[
 C_{2,2}(\Gamma^\infty/\Z_2)
 =C_{2,2}(\Gamma^\infty/U(1))=\frac1{\sqrt2}.
\]
\end{introtheorem}

We next consider the separation distance.
For a pyramid $\mathcal P$ with $0<C_{2,2}(\mathcal P)<\infty$, define
\[
 \ThetaPI(\mathcal P)
 \coloneqq\limsup_{\kappa\downarrow0}
 \frac{\Sep(\mathcal P;1/2,\kappa)^2}
      {C_{2,2}(\mathcal P)^2\log(1/\kappa)},
\]
where $\Sep(\mathcal P;\kappa_0,\kappa_1)$ denotes the separation distance of $\mathcal{P}$ with parameters $0<\kappa_0,\kappa_1\le1$
(see \Cref{ssec:sep} for the definition).
This coefficient is invariant under rescaling:
\[
 \ThetaPI(a\mathcal P)=\ThetaPI(\mathcal P)\qquad(a>0).
\]

\begin{introtheorem}\label{thm:intro-scale-separation}
\[
 \ThetaPI(\Gamma^\infty)=2,
 \qquad
 \ThetaPI(\Gamma^\infty/\Z_2)=4,
 \qquad
 \ThetaPI(\Gamma^\infty/U(1))=4.
\]
\end{introtheorem}

\Cref{thm:intro-scale-separation} implies that the Gaussian pyramid $\Gamma^\infty$ is non-similar to either quotient.
The two quotient pyramids have the same values of both $C_{2,2}$ and $\ThetaPI$, so distinguishing them requires a different argument.

\begin{introtheorem}\label{thm:intro-strict-containment}
\[
 \Gamma^\infty/U(1)\subsetneq\Gamma^\infty/\Z_2.
\]
\end{introtheorem}

A pyramid is a family of isomorphism classes of mm-spaces. Thus the strict containment in \Cref{thm:intro-strict-containment} means that $\Gamma^\infty/\Z_2$ contains mm-spaces that do not belong to $\Gamma^\infty/U(1)$. The metric half-line endowed with the half-normal distribution is one such space.

Since the two quotient pyramids have the same positive Poincar\'e constant, any similarity between them would have scale factor one.
Their strict containment therefore rules out similarity.
This proves \Cref{thm:main}.

The main difficulty is to exclude the half-normal line not only from the finite-dimensional $U(1)$ quotients, but also from their pyramidal limit.
Assuming membership yields a sequence of invariant Lipschitz observables whose distributions converge to the half-normal law.
Dimension-free moment estimates and a Gaussian spectral-gap argument force their centered squares to be approximated by invariant quadratic forms.
The limiting third moment then contradicts the constraint imposed by $U(1)$-invariance on these quadratic forms.
This provides an obstruction to membership in the limiting pyramid that is not detected by either the sharp Poincar\'e constant or $\ThetaPI$.

The results thus distinguish three standard geometric limits and show that identical values of these two invariants can coexist with different families of metric measure spaces in the limit.
They also suggest studying other Gaussian quotients through the interaction between invariant observables, their low-degree components, and moment constraints.

\Cref{sec:setup} introduces the three Gaussian pyramids and the auxiliary invariants used below. Section~2 develops the closure, moment, and Wiener-chaos estimates used later. Section~3 computes the sharp Poincar\'e constants of the two quotient pyramids by transferring the finite-dimensional calculations to their pyramid limits. Section~4 establishes the separation-distance estimates that distinguish the Gaussian pyramid from its quotients. Section~5 proves the strict containment by the half-normal witness and completes the pairwise non-similarity argument.

\section{Preliminaries}\label{sec:setup}

\subsection{Gromov's theory of mm-spaces}
The notions of mm-isomorphism, Lipschitz order, box distance and pyramid were introduced by Gromov \cite[\S 3$\frac12$]{gromov2007met}.
We refer to \cite{ShioyaBook} for the details.

For an mm-space $X$, we denote by $d_X$ and $\mu_X$ the metric and the measure of $X$, respectively.

\begin{definition}[mm-Isomorphism]
Two mm-spaces $X$ and $Y$ are \emph{mm-isomorphic} if there exists an isometry $f\colon\operatorname{supp}\mu_X\to\operatorname{supp}\mu_Y$ such that $f_*\mu_X=\mu_Y$,
where $\supp\mu$ is the support of a Borel measure $\mu$
and $f_*\mu_X$ is the pushforward of $\mu_X$ by $f$.
\end{definition}

Note that an mm-space $X$ is mm-isomorphic to $(\supp\mu_X,d_X,\mu_X)$.
We assume that the underlying measure $\mu_X$ of an mm-space $X$ is fully supported unless otherwise stated.

\begin{definition}[Box distance]
Let $I\coloneqq[0,1)$ be equipped with the Lebesgue measure $\mathcal L^1$. A Borel map $\varphi\colon I\to X$ satisfying $\varphi_*\mathcal L^1=\mu_X$ is called a \emph{parameter} of $X$. For a parameter $\varphi$ of $X$, set
\[
 \varphi^*d_X(s,t)\coloneqq d_X(\varphi(s),\varphi(t))
 \qquad(s,t\in I).
\]
For mm-spaces $X$ and $Y$, their \emph{box distance} $\Box(X,Y)$ is the infimum of all $\varepsilon\ge0$ for which there exist parameters $\varphi\colon I\to X$, $\psi\colon I\to Y$, and a Borel subset $I_0\subset I$ such that
\begin{align}
 &\abs{\varphi^*d_X(s,t)-\psi^*d_Y(s,t)}\le\varepsilon
 \qquad(s,t\in I_0),\tag{i}\\
 &\mathcal L^1(I_0)\ge1-\varepsilon.\tag{ii}
 \end{align}
\end{definition}

\begin{definition}[Lipschitz order]
For mm-spaces $X$ and $Y$, we say that \emph{$Y$ Lipschitz dominates $X$} and write $X\prec Y$ if there exists a  $1$-Lipschitz map $f : Y \to X$ such that $f_*\mu_Y = \mu_X$.
Such a map $f$ is called a \emph{domination map} and
the relation $\prec$ is called the \emph{Lipschitz order relation}.
\end{definition}

We identify an mm-space with its mm-isomorphism class when no confusion can arise.

\begin{definition}[Pyramid]
A \emph{pyramid} is a nonempty family $\mathcal P$ of mm-isomorphism classes of mm-spaces satisfying the following conditions:
\begin{enumerate}[label=\textup{(\roman*)},leftmargin=2.5em]
\item If $X\in\mathcal P$ and $Y\prec X$, then $Y\in\mathcal P$.
\item For any $X,Y\in\mathcal P$, there exists $Z\in\mathcal P$ such that $X\prec Z$ and $Y\prec Z$.
\item The family $\mathcal P$ is closed with respect to the box distance.
\end{enumerate}
\end{definition}

Gromov's original definition requires only (i) and (ii); condition (iii) was added in \cite{ShioyaBook,ShioyaCP} to make the space of pyramids Hausdorff.

For an mm-space $X$ and a pyramid $\mathcal P$, set
\[
 \Box(X,\mathcal P)\coloneqq\inf_{Y\in\mathcal P}\Box(X,Y).
\]
\begin{definition}[Weak convergence of pyramids]
Let $\mathcal P_n$ and $\mathcal P$ be pyramids. We say that $\{\mathcal P_n\}$ \emph{converges weakly} to $\mathcal P$ if the following conditions hold:
\begin{enumerate}[label=\textup{(\roman*)},leftmargin=2.5em]
\item For every $X\in\mathcal P$, we have
\[
 \lim_{n\to\infty}\Box(X,\mathcal P_n)=0.
\]
\item For every mm-space $X\notin\mathcal P$, we have
\[
 \liminf_{n\to\infty}\Box(X,\mathcal P_n)>0.
\]
\end{enumerate}
\end{definition}

Shioya \cite{ShioyaBook,ShioyaCP} constructed a metric on the set of pyramids compatible with the weak convergence.
The topology induced by the metric is called the weak topology.

For an mm-space $X$, let $\mathcal P_X := \{Y \mid Y \prec X\}$.
We call $\mathcal{P}_X$ the \emph{pyramid associated with $X$}.
The map $X \mapsto \mathcal{P}_X$ embeds the space of mm-isomorphism classes of mm-spaces with the concentration topology into the space of pyramids with the weak topology.
The space of pyramids is compact, and the image of this embedding is dense; hence it forms a compactification of the space of mm-isomorphism classes.
We call the space of pyramids the \emph{pyramidal compactification}.

For an mm-space $X$ and $a>0$, let $aX$ denote the space obtained by multiplying its distance by $a$. For a pyramid $\mathcal P$, set $a\mathcal P\coloneqq\{aX\mid X\in\mathcal P\}$.

\subsection{Gaussian Hopf quotients}
Let $\gamma^n\coloneqq N(0,I_n)$ be the \emph{standard Gaussian measure} on $\R^n$. Thus $\gamma^n=(\gamma^1)^{\otimes n}$, and its coordinates are independent and have distribution $N(0,1)$.
Set
\[
 \Gamma^n\coloneqq(\R^n,\norm{\cdot}_2,\gamma^n),
 \qquad
 \Gamma^n_{\C}\coloneqq(\C^n,\norm{\cdot}_2,\gamma^{2n}),
\]
where $\|\cdot\|_2$ indicates the $\ell^2$-norm.
We identify $\C^n$ with $\R^{2n}$ and take all its real and imaginary coordinates to be independent $N(0,1)$ random variables. The group $\Z_2=\{\pm1\}$ acts on $\Gamma^n$ by
\[
  \Z_2\times\R^n\longrightarrow\R^n,
  \qquad
  (\varepsilon,x)\longmapsto\varepsilon x,
\]
and $U(1)$ acts on $\Gamma^n_{\C}$ by
\[
  U(1)\times\C^n\longrightarrow\C^n,
  \qquad
  (e^{i\theta},z)\longmapsto e^{i\theta}z.
\]
For these two quotients, define the \emph{quotient distance} and \emph{quotient measure} by
\[
 d_{X/G}([x],[y])\coloneqq\inf_{g\in G}d_X(x,gy),
 \qquad
 \mu_{X/G}\coloneqq(\pi_G)_*\mu_X,
\]
where $\pi_G\colon X\to X/G$ is the \emph{quotient map}
and $(\pi_G)_*\mu_X$ is the pushforward of $\mu_X$ by $\pi_G$.

Let $\operatorname{pr}_n^{\R}\colon\R^{n+1}\to\R^n$ and $\operatorname{pr}_n^{\C}\colon\C^{n+1}\to\C^n$ be the coordinate projections that omit the last coordinate. They are measure-preserving $1$-Lipschitz maps and commute with the actions of $\Z_2$ and $U(1)$, respectively. They therefore induce maps
\[
 \overline{\operatorname{pr}}{}_n^{\R}\colon\R^{n+1}/\Z_2\to\R^n/\Z_2,
 \qquad
 \overline{\operatorname{pr}}{}_n^{\C}\colon\C^{n+1}/U(1)\to\C^n/U(1)
\]
given by
\[
 \overline{\operatorname{pr}}{}_n^{\R}([x]) \coloneqq[\operatorname{pr}_n^{\R}(x)], \qquad
 \overline{\operatorname{pr}}{}_n^{\C}([z]) \coloneqq[\operatorname{pr}_n^{\C}(z)].
\]
Since $\operatorname{pr}_n^{\R}$ commutes with the $\Z_2$-action and $\operatorname{pr}_n^{\C}$ commutes with the $U(1)$-action, the orbits on the right-hand sides do not depend on the representatives of $[x]$ and $[z]$. Thus both induced maps are well defined. They are measure-preserving and $1$-Lipschitz, and hence
\[
 \Gamma^n\prec\Gamma^{n+1},\qquad
 \Gamma^n/\Z_2\prec\Gamma^{n+1}/\Z_2,\qquad
 \Gamma^n_{\C}/U(1)\prec\Gamma^{n+1}_{\C}/U(1).
\]
The associated pyramids of the finite-dimensional models therefore form increasing sequences. We denote the box closures of the unions of these sequences by
\begin{align*}
 \Gamma^\infty
 &\coloneqq\overline{\bigcup_{n\ge1}\mathcal P_{\Gamma^n}},\\
 \Gamma^\infty/\Z_2
 &\coloneqq\overline{\bigcup_{n\ge1}\mathcal P_{\Gamma^n/\Z_2}},\\
\Gamma^\infty/U(1)
 &\coloneqq\overline{\bigcup_{n\ge1}\mathcal P_{\Gamma^n_{\C}/U(1)}}.
\end{align*}
The associated-pyramid sequences converge weakly to the corresponding pyramids defined above.
$\Gamma^\infty$ is called the \emph{infinite-dimensional virtual standard Gaussian space}.  In this paper, we call it the \emph{Gaussian pyramid} for simplicity.

\subsection{Asymptotic Lipschitz constants and Poincar\'e constants}
For a metric space $X$, let $\Lip(X)$ be the set of real-valued Lipschitz functions on $X$, and let $\Lipb(X)$ be the subset of bounded ones. Set
\[
 \Lip_1(X)\coloneqq\{f\colon X\to\R\mid\abs{f(x)-f(y)}\le d_X(x,y)
 \text{ for all }x,y\in X\}.
\]
For $x\in X$ and $r>0$, write $B_r(x)\coloneqq\{y\in X\mid d_X(x,y)<r\}$. For $f\in\Lip(X)$ and $x\in X$, define the \emph{asymptotic Lipschitz constant} of $f$ by
\begin{equation*}
 \lipa(f)(x)\coloneqq\lim_{r\downarrow0}\Lip\bigl(f|_{B_r(x)}\bigr)
 =\inf_{r>0}\Lip\bigl(f|_{B_r(x)}\bigr).
\end{equation*}
Here
\[
 \Lip(f|_A)\coloneqq\sup_{u,v\in A,\ u\ne v}\frac{\abs{f(u)-f(v)}}{d_X(u,v)},
\]
with value zero when $A$ has at most one point. In particular, $\lipa(f)=0$ at every isolated point.

For a measure space $(X,\mu)$ and $1\le p<\infty$, write
\[
 \norm{f}_{L^p(\mu)}
 \coloneqq\left(\int_X|f|^p\dd\mu\right)^{1/p},
 \qquad
 \inner{f}{g}_{L^2(\mu)}
 \coloneqq\int_Xfg\dd\mu.
\]
When the measure is clear from the context, we abbreviate $\norm f_{L^p(\mu)}$ to $\norm f_p$. If $\mu$ is a probability measure and $f\in L^2(\mu)$, define the \emph{variance} by
\[
 \Var_\mu(f)\coloneqq\int_X\left(f-\int_Xf\dd\mu\right)^2\dd\mu,
\]
and, for an mm-space $X$, write $\Var_X(f)\coloneqq\Var_{\mu_X}(f)$. 

\begin{definition}[$(2,2)$-Poincar\'e constant]
Define the \emph{$(2,2)$-Poincar\'e constant} of an mm-space $X$ as
\begin{equation}
 C_{2,2}(X)\coloneqq\inf\left\{C\in[0,\infty)\mid
 \Var_X(f)\le C^2\int_X\lipa(f)^2\dd\mu_X\quad(f\in\Lipb(X))\right\}. \label{eq:C22-def}
\end{equation}
The infimum of the empty set is understood to be $+\infty$.
For a pyramid $\mathcal P$, we use the standard extension due to Esaki--Kazukawa--Mitsuishi \cite{EKM}:
\[
 C_{2,2}(\mathcal P)\coloneqq\sup_{X\in\mathcal P}C_{2,2}(X).
\]
\end{definition}
$C_{2,2}$ is $1$-homogeneous:
\[
 C_{2,2}(a\mathcal P)=aC_{2,2}(\mathcal P) \qquad (a > 0).
\]

\subsection{Separation distance}
\label{ssec:sep}
The separation distance is defined by Gromov \cite{gromov2007met}.
\begin{definition}[Separation distance]
For $0 < \alpha,\beta \le 1$, define the \emph{separation distance} of an mm-space $X$ by
\[
 \begin{aligned}
 \Sep(X;\alpha,\beta)
 \coloneqq\sup\{d_X(A,B)\mid {}&A,B\subset X\text{ are Borel sets},\\
 &\mu_X(A)\ge\alpha,\ \mu_X(B)\ge\beta\},
 \end{aligned}
\]
where
\[
 d_X(A,B)\coloneqq\inf\{d_X(x,y)\mid x\in A,\ y\in B\}.
\]
\end{definition}

The separation distance is extended to a pyramid by \cite[Definition~4.3]{OzawaShioya}.

\begin{definition}[Separation distance for a pyramid]
For $0 < \alpha,\beta \le 1$, the \emph{separation distance} of a pyramid $\mathcal{P}$ is defined by
\[
 \Sep(\mathcal P;\alpha,\beta)
 \coloneqq\lim_{\delta\downarrow0}\sup_{X\in\mathcal P}
       \Sep(X;\alpha-\delta,\beta-\delta).
\]
\end{definition}

For every $X\in\mathcal P$, the quantity $\Sep(X;\alpha-\delta,\beta-\delta)$ is nonincreasing as $\delta\downarrow0$. The supremum over $X$ is therefore also nonincreasing, so the limit on the right-hand side exists as an extended real number, possibly $+\infty$.

For every mm-space $X$, \cite[Proposition~4.4]{OzawaShioya} gives
\[
 \Sep(\mathcal P_X;\alpha,\beta) = \Sep(X;\alpha,\beta).
\]
The separation distance is $1$-homogeneous:
\[
 \Sep(a\mathcal P;\alpha,\beta)=a\Sep(\mathcal P;\alpha,\beta). 
\]

\section{Closure and Gaussian-analysis lemmas}

\subsection{Quotient maps and invariant observables}

\begin{lemma}\label{lem:quotient}
Let a compact group $G$ act on an mm-space $X$ by measure-preserving isometries. The quotient map $\pi_G\colon X\to X/G$ is measure-preserving and $1$-Lipschitz. Therefore, if $\bar f\in\Lip_1(X/G)$, then
\[
 f\coloneqq\bar f\circ\pi_G\in\Lip_1(X),
 \qquad
 f_*\mu_X=\bar f_*\mu_{X/G}.
\]
Conversely, if $f\in\Lip_1(X)$ is $G$-invariant, then there exists a unique $\bar f\colon X/G\to\R$ such that $f=\bar f\circ\pi_G$, and this $\bar f$ is also $1$-Lipschitz.

Suppose in addition that $G$ acts on mm-spaces $X$ and $Y$ by measure-preserving isometries, and let $\pi_X\colon X\to X/G$ and $\pi_Y\colon Y\to Y/G$ be the respective quotient maps. If a measure-preserving $1$-Lipschitz map $T\colon X\to Y$ is $G$-equivariant, that is, $T(gx)=gT(x)$, then
\[
 \bar T\colon X/G\longrightarrow Y/G,
 \qquad
 \bar T([x])\coloneqq[T(x)]
\]
is independent of the choice of representative and is a measure-preserving $1$-Lipschitz map.
\end{lemma}

\begin{proof}
The definition of the quotient measure shows that $\pi_G$ is measure-preserving, and
\[
 d_{X/G}(\pi_G(x),\pi_G(y))
 =\inf_{g\in G}d_X(x,gy)\le d_X(x,y),
\]
so it is $1$-Lipschitz. Thus $\bar f\circ\pi_G$ is also $1$-Lipschitz, and
\[
 (\bar f\circ\pi_G)_*\mu_X
 =(\bar f)_*(\pi_G)_*\mu_X
 =(\bar f)_*\mu_{X/G}.
\]
Conversely, for a $G$-invariant function $f$, set $\bar f([x])\coloneqq f(x)$. This is independent of the representative. For every $g\in G$,
\[
 |\bar f([x])-\bar f([y])|
 =|f(x)-f(gy)|\le d_X(x,gy).
\]
Taking the infimum over $g$ shows that $\bar f$ is $1$-Lipschitz.

If $T$ is $G$-equivariant, then $[x]=[y]$ implies $[T(x)]=[T(y)]$, so $\bar T$ is independent of the representative. Moreover,
\begin{align*}
 d_{Y/G}(\bar T([x]),\bar T([y]))
 &=\inf_{g\in G}d_Y(T(x),gT(y))\\
 &=\inf_{g\in G}d_Y(T(x),T(gy))
 \le d_{X/G}([x],[y]).
\end{align*}
Thus $\bar T$ is $1$-Lipschitz. Finally, $\bar T\circ\pi_X=\pi_Y\circ T$ gives
\[
 (\bar T)_*\mu_{X/G}
 =(\bar T)_*(\pi_X)_*\mu_X
 =(\pi_Y)_*T_*\mu_X
 =(\pi_Y)_*\mu_Y
 =\mu_{Y/G}.
\]
\end{proof}

\subsection{Dimension-free moment convergence}

We establish a dimension-free criterion that upgrades weak convergence of Lipschitz observables to convergence of moments of every positive order.
The Gaussian isoperimetric inequality provides the required tail estimates.

Write $\Phi$ for the \emph{distribution function} of the \emph{standard normal distribution} and
\[
 \varphi(t)\coloneqq(2\pi)^{-1/2}e^{-t^2/2}
\]
for its \emph{density}. For a set $A\subset\R^n$ and $x\in\R^n$, write
\[
 d_{\R^n}(x,A)\coloneqq\inf_{a\in A}\norm{x-a}_2,
\]
and, for $t>0$, set
$A_t\coloneqq\{x\in\R^n\mid d_{\R^n}(x,A)<t\}$.

In what follows, $\N\coloneqq\{1,2,\dots\}$ and $\mathbb N_0\coloneqq\N\cup\{0\}=\{0,1,2,\dots\}$.
We write $\Rightarrow$ for weak convergence of probability measures.

\begin{theorem}[Gaussian isoperimetric inequality {\cite{Borell,SudakovTsirelson}}]\label{thm:gauss-iso}
For every measurable set $A\subset\R^n$ and every $t>0$,
\begin{equation*}
 \gamma^n(A_t)\ge
 \Phi\bigl(\Phi^{-1}(\gamma^n(A))+t\bigr)
\end{equation*}
holds.
Here we use the conventions $\Phi^{-1}(0)=-\infty$ and $\Phi^{-1}(1)=+\infty$.
\end{theorem}

\begin{lemma}\label{lem:gauss-conc}
Let $u\in\Lip_1(\R^n)$, and let $m_u$ be a \emph{median} of $u_*\gamma^n$, that is, a real number satisfying
\[
 \gamma^n(u\le m_u)\ge\frac12,
 \qquad
 \gamma^n(u\ge m_u)\ge\frac12.
\]
Then, for every $t>0$,
\begin{equation*}
 \gamma^n(\abs{u-m_u}>t)
 \le2(1-\Phi(t))\le2e^{-t^2/2}.
\end{equation*}
In particular, the right-hand side does not depend on the dimension $n$.
\end{lemma}

\begin{proof}
Set $A\coloneqq\{u\le m_u\}$. Then $\gamma^n(A)\ge1/2$. If $u(x)>m_u+t$, then, for every $a\in A$,
\[
 t<u(x)-u(a)\le\norm{x-a}_2,
\]
and hence $x\notin A_t$. Therefore, \Cref{thm:gauss-iso} gives
\[
 \gamma^n(u>m_u+t)
 \le1-\gamma^n(A_t)
 \le1-\Phi(t).
\]
Since $-m_u$ is a median of $-u$, applying the same argument to $-u$ gives $\gamma^n(u<m_u-t)\le1-\Phi(t)$. Adding the two inequalities and using the standard normal tail bound $1-\Phi(t)\le e^{-t^2/2}$ \cite{Gordon} proves the assertion.
\end{proof}

\begin{lemma}\label{lem:lsc-integral}
Suppose that Borel probability measures $\nu_j,\nu$ on $\R$ satisfy $\nu_j\Rightarrow\nu$ as $j\to\infty$, and let $g\colon\R\to[0,\infty]$ be lower semicontinuous. Then
\begin{equation*}
 \int_\R g\dd\nu\le\liminf_{j\to\infty}\int_\R g\dd\nu_j.
\end{equation*}
\end{lemma}

\begin{proof}
We derive the required integral form from the set version of the Portmanteau theorem by using the layer-cake representation. For $t\ge0$, set $U_t\coloneqq\{g>t\}$. The lower semicontinuity of $g$ implies that $U_t$ is open, and the Portmanteau theorem (cf.~\cite[Lemma~1.13(3)]{ShioyaBook}) gives $\nu(U_t)\le\liminf_{j\to\infty}\nu_j(U_t)$. Combining the identity $\int_\R g\dd\nu=\int_0^\infty\nu(U_t)\dd t$ for nonnegative measurable functions with Fatou's lemma, we obtain
\[
 \int_\R g\dd\nu
 \le\int_0^\infty\liminf_{j\to\infty}\nu_j(U_t)\dd t
 \le\liminf_{j\to\infty}\int_0^\infty\nu_j(U_t)\dd t
 =\liminf_{j\to\infty}\int_\R g\dd\nu_j.
\]
\end{proof}

\begin{lemma}\label{lem:UI}
Let $n_j\in\N$ and $F_j\in\Lip_1(\R^{n_j})$, and suppose that
\begin{equation}
 (F_j)_*\gamma^{n_j}\Rightarrow\nu\quad(j\to\infty).           \label{eq:Fweak}
\end{equation}
Then, for every $p>0$,
\begin{equation}
 \sup_j\int_{\R^{n_j}}\abs{F_j}^p\dd\gamma^{n_j}<\infty.         \label{eq:uniform-mom}
\end{equation}
Moreover, if $\psi\colon\R\to\R$ is continuous and satisfies
$\abs{\psi(t)}\le C(1+\abs t^p)$ for some $C>0$ and $p>0$, then
\begin{equation}
 \int_{\R^{n_j}}\psi(F_j)\dd\gamma^{n_j}
 \longrightarrow\int_\R\psi(t)\dd\nu(t)\quad(j\to\infty).    \label{eq:moment-conv}
\end{equation}
In particular, every absolute moment of positive order and the expectation of every fixed polynomial converge, and $\nu$ has moments of all finite orders.
\end{lemma}

\begin{proof}
Choose $R>0$ such that $\nu((-R,R))>3/4$.
Since $(-R,R)$ is open, \eqref{eq:Fweak} and the Portmanteau theorem give
\[
 \liminf_{j\to\infty}\gamma^{n_j}(|F_j|<R)
 \ge\nu((-R,R))>\frac34.
\]
Thus, apart from finitely many initial terms,
$\gamma^{n_j}(|F_j|\le R)>1/2$. If $m_j$ is any median of $F_j$, then
$m_j\in[-R,R]$. Indeed, if $m_j>R$, then
$\gamma^{n_j}(F_j\ge m_j)\le\gamma^{n_j}(F_j>R)<1/2$, contrary to the definition of a median.
The case $m_j<-R$ is analogous. Enlarging the bound to include the finitely many initial medians, we obtain
\begin{equation}
 M_0\coloneqq\sup_j|m_j|<\infty.                                        \label{eq:med-bound}
\end{equation}

By \Cref{lem:gauss-conc},
\[
 \gamma^{n_j}(|F_j-m_j|>t)\le2e^{-t^2/2}.
\]
Applying the identity
\[
 \mathbb EY^p=p\int_0^\infty t^{p-1}\mathbb P(Y>t)\dd t
\]
for a nonnegative random variable $Y$ and $p>0$ to $Y=|F_j-m_j|$ gives
\[
 \sup_j\mathbb E|F_j-m_j|^p
 \le2p\int_0^\infty t^{p-1}e^{-t^2/2}\dd t<\infty.
\]
Since
$(a+b)^p\le2^{\max\{p-1,0\}}(a^p+b^p)$ for $a,b\ge0$, combining this estimate with \eqref{eq:med-bound} proves \eqref{eq:uniform-mom}.

We next prove \eqref{eq:moment-conv}. Fix a continuous function $\psi$ satisfying $|\psi(t)|\le C(1+|t|^p)$ for some $C>0$ and $p>0$, and take $s>p$. Applying \eqref{eq:uniform-mom} with exponent $s$ gives
\[
 K_s\coloneqq\sup_j\int_{\R^{n_j}}\abs{F_j}^s\dd\gamma^{n_j}<\infty.
\]
For $M>0$, set
\[
 \tau_M(t)\coloneqq\max\{-M,\min\{t,M\}\},
 \qquad
 \psi_M(t)\coloneqq\psi(\tau_M(t)).
\]
Since $\psi_M$ is bounded and continuous, the weak convergence \eqref{eq:Fweak} gives
\begin{equation}
 \int_{\R^{n_j}}\psi_M(F_j)\dd\gamma^{n_j}
 \longrightarrow \int_\R\psi_M(t)\dd\nu(t)\quad(j\to\infty). \label{eq:truncated-psi}
\end{equation}
On the other hand, on $|F_j|>M$ we have $1\le M^{-s}|F_j|^s$ and
$|F_j|^p\le M^{p-s}|F_j|^s$. Therefore,
\begin{equation}
 \sup_j\int_{\abs{F_j}>M}(1+\abs{F_j}^p)\dd\gamma^{n_j}
 \le (M^{-s}+M^{p-s})K_s
 \xrightarrow[M\to\infty]{}0.                                     \label{eq:tail-unif}
\end{equation}

Apply \Cref{lem:lsc-integral} to the nonnegative lower semicontinuous function $g(t)\coloneqq|t|^s$. Then
\[
 \int_\R |t|^s\dd\nu(t)\le K_s<\infty.
\]
Since the right-hand side is finite and independent of $M$, the same estimate gives
\[
 \int_{|t|>M}(1+|t|^p)\dd\nu(t)\xrightarrow[M\to\infty]{}0.
\]
Furthermore, $\psi_M(t)=\psi(t)$ when $|t|\le M$, while
\[
 |\psi(t)-\psi_M(t)|
 \le 2C(1+|t|^p)
\]
when $|t|>M$. Combining \eqref{eq:tail-unif} and the tail estimate for the limiting measure with \eqref{eq:truncated-psi}, and letting first $j\to\infty$ for fixed $M$ and then $M\to\infty$, proves \eqref{eq:moment-conv}.
\end{proof}

The following are the standard formulas obtained by integration by parts against the standard normal density \cite{Janson}.

\begin{lemma}\label{lem:normal-mom}
Let $\xi\sim N(0,1)$. For every integer $k\ge0$,
\begin{equation*}
 \mathbb E\xi^{2k+1}=0,
 \qquad
 \mathbb E\xi^{2k}=(2k-1)!!\coloneqq1\cdot3\cdots(2k-1),
\end{equation*}
where $(-1)!!\coloneqq1$. In particular,
\[
 \mathbb E\xi^2=1,
 \quad \mathbb E\xi^4=3,
 \quad \mathbb E\xi^6=15.
\]
\end{lemma}

\subsection{Low-degree component estimates via Wiener chaos}

We collect Wiener-chaos estimates used below and later in the proof of \Cref{thm:H1-not-U1}. We refer to \cite[Chapter 2]{Janson} for Wiener chaos.

Write points of $\R^n$ as $x=(x_1,\dots,x_n)$. For a locally integrable function $u$ on $\R^n$, its \emph{weak derivative} $\partial_i u$ is the locally integrable function satisfying
\[
 \int_{\R^n}u\,\partial_i\varphi\dd x
 =-\int_{\R^n}(\partial_i u)\varphi\dd x
\]
for every $\varphi\in C_c^\infty(\R^n)$. For locally Lipschitz functions, weak derivatives agree almost everywhere with the classical derivatives. Let $W^{1,2}(\gamma^n)$ be the \emph{Gaussian Sobolev space} of all $u\in L^2(\gamma^n)$ such that every weak derivative $\partial_i u$ belongs to $L^2(\gamma^n)$, with norm
\[
 \norm{u}_{W^{1,2}(\gamma^n)}^2
 \coloneqq\norm{u}_{L^2(\gamma^n)}^2
   +\sum_{i=1}^n\norm{\partial_i u}_{L^2(\gamma^n)}^2.
\]

For $u\in W^{1,2}(\gamma^n)$, write
\[
 \norm{\nabla u(x)}_2\coloneqq\left(\sum_{i=1}^n|\partial_i u(x)|^2\right)^{1/2}
\]
for $\gamma^n$-almost every $x\in\R^n$. The notation $\norm{\nabla u}_2$ without $x$ denotes this pointwise function. Define the \emph{$L^2$-norm of the gradient} by
\[
 \norm{\nabla u}_{L^2(\gamma^n;\R^n)}
 \coloneqq\left(\int_{\R^n}\norm{\nabla u(x)}_2^2\dd\gamma^n(x)\right)^{1/2}
 =\left(\sum_{i=1}^n\norm{\partial_i u}_{L^2(\gamma^n)}^2\right)^{1/2}.
\]

Define the \emph{one-variable Hermite polynomials} $\mathrm{He}_k$ by the generating function
\[
 \exp\left(tx-\frac{t^2}{2}\right)
 =\sum_{k=0}^\infty\mathrm{He}_k(x)\frac{t^k}{k!}.
\]
For a multi-index $\alpha=(\alpha_1,\dots,\alpha_n)\in\mathbb N_0^n$, set
\[
 |\alpha|\coloneqq\sum_{i=1}^n\alpha_i,\qquad
 \alpha!\coloneqq\prod_{i=1}^n\alpha_i!,\qquad
 \mathrm{He}_\alpha(x)\coloneqq\prod_{i=1}^n\mathrm{He}_{\alpha_i}(x_i).
\]
We call $\mathrm{He}_\alpha$ the \emph{multivariate Hermite polynomial} associated with $\alpha$. The space
\[
 \mathcal H_m\coloneqq
 \operatorname{span}\{\mathrm{He}_\alpha\mid|\alpha|=m\}
 \subset L^2(\gamma^n)
\]
is called the \emph{$m$th Wiener chaos}. The Hermite polynomials form an orthogonal basis of $L^2(\gamma^n)$ \cite[Theorem~1.1.1]{Nualart}; hence
\begin{equation}
 L^2(\gamma^n)=\bigoplus_{m=0}^\infty\mathcal H_m.                 \label{eq:chaos-decomposition}
\end{equation}
Write
\[
 P_m\colon L^2(\gamma^n)\longrightarrow\mathcal H_m
\]
for the orthogonal projection onto $\mathcal H_m$. Thus, for every $u\in L^2(\gamma^n)$,
\begin{equation}
 u=\sum_{m=0}^\infty P_mu,\qquad
 \norm{u}_2^2=\sum_{m=0}^\infty\norm{P_mu}_2^2,                  \label{eq:chaos-Parseval}
\end{equation}
where the first series converges in $L^2(\gamma^n)$. The term $P_mu$ extracts the component of $u$ of Hermite degree $m$. More concretely, it is the element of $\mathcal H_m$ closest to $u$ in $L^2$.

In the one-variable generating function,
\[
 \exp\left(t(-x)-\frac{t^2}{2}\right)
 =\exp\left((-t)x-\frac{(-t)^2}{2}\right).
\]
Comparing the coefficients of $t^k/k!$ gives $\mathrm{He}_k(-x)=(-1)^k\mathrm{He}_k(x)$. The product definition therefore yields
\begin{equation}
 \mathrm{He}_\alpha(-x)=(-1)^{|\alpha|}\mathrm{He}_\alpha(x).      \label{eq:Hermite-parity}
\end{equation}
Thus functions in $\mathcal H_m$ are even when $m$ is even and odd when $m$ is odd.

Write $\operatorname{Sym}(n)$ for the set of real symmetric $n\times n$ matrices. The space of \emph{centered quadratic forms} is
\begin{equation*}
 \mathcal Q_n\coloneqq\{q_A:A\in\operatorname{Sym}(n)\},
 \qquad
 q_A(x)\coloneqq x^{\mathsf T}Ax-\tr A.
\end{equation*}
Here $x$ is a column vector and $x^{\mathsf T}$ is its transpose.

\begin{lemma}\label{lem:quad-trace}
\[
 \mathcal H_0=\operatorname{span}\{1\},\qquad
 \mathcal H_1=\operatorname{span}\{x_1,\dots,x_n\},\qquad
 \mathcal H_2=\mathcal Q_n.
\]
Moreover, for every $A,B\in\operatorname{Sym}(n)$,
\begin{equation*}
 \int_{\R^n}q_Aq_B\dd\gamma^n=2\tr(AB).
\end{equation*}
holds.
The map $\operatorname{Sym}(n)\ni A\mapsto q_A\in\mathcal Q_n$ is bijective.
\end{lemma}

\begin{proof}
The coefficients of degrees zero, one, and two in the one-variable generating function give $\mathrm{He}_0(x)=1$, $\mathrm{He}_1(x)=x$, and $\mathrm{He}_2(x)=x^2-1$. Thus $\mathcal H_0$ and $\mathcal H_1$ are as stated, and
\[
 \mathcal H_2
 =\operatorname{span}\{x_i^2-1\mid1\le i\le n\}
  +\operatorname{span}\{x_ix_j\mid1\le i<j\le n\}.
\]
On the other hand, separating the diagonal and off-diagonal terms gives
\[
 x^{\mathsf T}Ax=\sum_{i,j=1}^nA_{ij}x_ix_j=\sum_{i=1}^nA_{ii}x_i^2+\sum_{i\ne j}A_{ij}x_ix_j.
\]
By the symmetry $A_{ij}=A_{ji}$,
\[
 \sum_{i\ne j}A_{ij}x_ix_j=\sum_{1\le i<j\le n}(A_{ij}+A_{ji})x_ix_j=2\sum_{1\le i<j\le n}A_{ij}x_ix_j.
\]
Since $\tr A=\sum_{i=1}^nA_{ii}$, we have
\[
 q_A(x)=x^{\mathsf T}Ax-\tr A=\sum_{i=1}^nA_{ii}(x_i^2-1)+2\sum_{1\le i<j\le n}A_{ij}x_ix_j.
\]
Therefore, $\mathcal Q_n\subset\mathcal H_2$. Conversely, if $A_{ii}=1$ and all other entries are zero, then $x_i^2-1=q_A$. If $A_{ij}=A_{ji}=1/2$ and all other entries are zero, then $x_ix_j=q_A$. This proves $\mathcal H_2=\mathcal Q_n$.

Take any $A,B\in\operatorname{Sym}(n)$ and regard $x=(x_1,\dots,x_n)$ as a random variable with distribution $\gamma^n$. Writing $\delta_{ij}$ for the Kronecker delta, we have $\mathbb E[x_ix_j]=\delta_{ij}$ and hence
\[
 q_A(x)=\sum_{i,j=1}^nA_{ij}(x_ix_j-\delta_{ij}).
\]
Independence and \Cref{lem:normal-mom} imply that, for every $i,j,k,l$,
\[
 \mathbb E[x_ix_jx_kx_l]=\delta_{ij}\delta_{kl}+\delta_{ik}\delta_{jl}+\delta_{il}\delta_{jk}.
\]
Indeed, if all four indices are equal, both sides equal $3$. If two distinct values each occur twice, both sides equal $1$. In every other case, at least one independent coordinate appears to an odd power on the left-hand side, while all three products on the right-hand side vanish, so both sides equal $0$. Therefore,
\[
 \mathbb E[(x_ix_j-\delta_{ij})(x_kx_l-\delta_{kl})]
 =\delta_{ik}\delta_{jl}+\delta_{il}\delta_{jk}.
\]
Substituting this identity into the representations of $q_A$ and $q_B$ and using the symmetry of $B$, we obtain
\begin{align*}
 \int_{\R^n}q_Aq_B\dd\gamma^n
 &=\sum_{i,j,k,l=1}^nA_{ij}B_{kl}(\delta_{ik}\delta_{jl}+\delta_{il}\delta_{jk})\\
 &=\sum_{i,j=1}^nA_{ij}B_{ij}+\sum_{i,j=1}^nA_{ij}B_{ji}
 =2\sum_{i,j=1}^nA_{ij}B_{ij}=2\tr(AB).
\end{align*}
This proves the integral identity. In particular, taking $B=A$ gives
\[
 \norm{q_A}_2^2=2\tr(A^2)=2\sum_{i,j=1}^nA_{ij}^2.
\]
Thus $q_A=0$ implies $A=0$, so the map $A\mapsto q_A$ is injective. Its image is $\mathcal Q_n$ by definition, so it is also surjective. Moreover, $\mathcal Q_n=\mathcal H_2$ is a finite-dimensional closed subspace.
\end{proof}

\begin{lemma}\label{lem:chaos-energy}
For every $u,v\in W^{1,2}(\gamma^n)$,
\begin{align}
 \int_{\R^n}\norm{\nabla u}_2^2\dd\gamma^n
 &=\sum_{m=1}^\infty m\norm{P_mu}_2^2,                           \label{eq:chaos-energy}\\
 \int_{\R^n}\inner{\nabla u}{\nabla v}\dd\gamma^n
 &=\sum_{m=1}^\infty m\inner{P_mu}{P_mv}_{L^2(\gamma^n)}.         \label{eq:chaos-energy-polarized}
\end{align}
\end{lemma}

\begin{proof}
For the coordinate Gaussian family on $(\R^n,\gamma^n)$, the Malliavin derivative and the chaos projections in \cite[Proposition~1.2.2]{Nualart} identify with $\nabla$ and $P_m$, respectively. Thus the first identity follows from that proposition, and the second identity follows by polarization.
\end{proof}

For an even function with zero mean, the first possible nonzero Hermite component has degree two. Orthogonality to centered quadratic forms raises this degree to four.

\begin{lemma}\label{lem:gap-hierarchy}
Let $u\in W^{1,2}(\gamma^n)$. Then the following hold.
\begin{enumerate}[label=\textup{(\arabic*)},leftmargin=2.5em]
\item If $u$ is even and $\mathbb Eu=0$, then
\begin{equation*}
 2\norm{u}_2^2\le\int\norm{\nabla u}_2^2\dd\gamma^n.
\end{equation*}
\item If $u$ is even, $\mathbb Eu=0$, and $u\perp\mathcal Q_n$, then
\begin{equation*}
 4\norm{u}_2^2\le\int\norm{\nabla u}_2^2\dd\gamma^n.
\end{equation*}
\end{enumerate}
\end{lemma}

\begin{proof}
Suppose that, for some $k\ge1$, $P_mu=0$ for every $0\le m<k$. Applying \eqref{eq:chaos-energy} and \eqref{eq:chaos-Parseval} to this $u$ gives
\begin{align}
 \int_{\R^n}\norm{\nabla u}_2^2\dd\gamma^n
 &=\sum_{m=k}^\infty m\norm{P_mu}_2^2\\
 &\ge k\sum_{m=k}^\infty\norm{P_mu}_2^2
 =k\norm{u}_2^2.                                                  \label{eq:chaos-gap-general}
\end{align}
It is therefore enough in each case to verify that the assumptions eliminate the components of degree less than $k$.

If $u$ is even, then for every odd $m$ and every $h\in\mathcal H_m$, the change of variables $x\mapsto-x$ and \eqref{eq:Hermite-parity} give
\[
 \inner{u}{h}_{L^2(\gamma^n)}
 =\int_{\R^n}u(-x)h(-x)\dd\gamma^n(x)
 =-\inner{u}{h}_{L^2(\gamma^n)}=0.
\]
Thus $P_mu=0$ for every odd $m$. If also $\mathbb Eu=0$, then $P_0u=0$, so all components of degree less than $2$ vanish. Taking $k=2$ in \eqref{eq:chaos-gap-general} proves part~(1).

Finally, if $u$ is even, its odd-degree components vanish. The zero-mean condition gives $P_0u=0$, and $\mathcal H_2=\mathcal Q_n$ in \Cref{lem:quad-trace} shows that $u\perp\mathcal Q_n$ is equivalent to $P_2u=0$. Therefore, $P_0u=P_1u=P_2u=P_3u=0$. Applying \eqref{eq:chaos-gap-general} with $k=4$ proves the assertion.
\end{proof}

\begin{lemma}\label{lem:Lip-Sobolev}
If $u\colon\R^n\to\R$ is $L$-Lipschitz for a constant $L \ge 0$, then $u\in W^{1,2}(\gamma^n)$ and
\[
 \norm{\nabla u(x)}_2\le L\qquad\text{for almost every }x\in\R^n.
\]
Moreover, $u^2\in W^{1,2}(\gamma^n)$ and
\begin{equation*}
 \nabla(u^2)=2u\nabla u\qquad\gamma^n\text{-almost everywhere}.
\end{equation*}
\end{lemma}

\begin{proof}
Since $u$ has at most linear growth, $u\in L^4(\gamma^n)$. The assertion
then follows from the standard Sobolev regularity of Lipschitz functions
and the Sobolev chain rule \cite[Chapter~4]{EvansGariepy}.
\end{proof}

\subsection{\texorpdfstring{$\chi$}{chi} distributions and radial maps on quotients}

\begin{definition}[$\chi$ distribution]\label{def:chi}
Let $n\in\N$. Using the \emph{radial map}
\[
 \rho_n\colon\R^n\to[0,\infty),\qquad \rho_n(x)\coloneqq\norm{x}_2,
\]
define
\begin{equation*}
 \chi_n\coloneqq(\rho_n)_*\gamma^n.
\end{equation*}
This is called the \emph{$\chi$ distribution} with $n$ degrees of freedom.
Equivalently, if $\xi_1,\dots,\xi_n$ are independent standard normal variables, then the distribution of $\sqrt{\xi_1^2+\cdots+\xi_n^2}$ is $\chi_n$.
\end{definition}

For $a>0$, let
\[
 \Gamma(a)\coloneqq\int_0^\infty u^{a-1}e^{-u}\dd u
\]
denote \emph{Euler's Gamma function}.

The following facts are well known. See, for example, \cite[Chap.~18]{JohnsonKotzBalakrishnan}.

\begin{lemma}\label{lem:chi}
Let $V=(V_1,\dots,V_n)\sim\gamma^n$ and $S_n\coloneqq\norm V_2$.
The distribution of $S_n$ is $\chi_n$, and it has density
\begin{equation*}
 f_n(r)=\frac1{2^{n/2-1}\Gamma(n/2)}r^{n-1}e^{-r^2/2}
 \qquad(r>0)
\end{equation*}
on $[0,\infty)$. Moreover,
\begin{equation*}
 \mathbb ES_n^2=n,
 \qquad
 \Var(S_n^2)=2n.
\end{equation*}
In particular, for $r\ge0$,
\begin{align}
 \mathbb P(S_1\le r)&=2\Phi(r)-1,                                 \label{eq:half-normal-cdf}\\
 \mathbb P(S_2\le r)&=1-e^{-r^2/2}.                               \label{eq:rayleigh-cdf}
\end{align}
\end{lemma}

\begin{lemma}\label{lem:radial-quotient}
Let $m\ge1$. Since the radial functions are invariant under the group actions, they define functions on the quotients by
\[
 \bar\rho_m^{\R}\colon\R^m/\Z_2\to[0,\infty),\quad \bar\rho_m^{\R}([x])=\norm x_2,
\]
\[
 \bar\rho_m^{\C}\colon\C^m/U(1)\to[0,\infty),\quad \bar\rho_m^{\C}([z])=\norm z_2.
\]
These functions are $1$-Lipschitz. If $\pi_{\Z_2}\colon\R^m\to\R^m/\Z_2$ and $\pi_{U(1)}\colon\C^m\to\C^m/U(1)$ are the quotient maps, then $\bar\rho_m^{\R}\circ\pi_{\Z_2}=\rho_m$ and $\bar\rho_m^{\C}\circ\pi_{U(1)}=\rho_{2m}$, and
\begin{align*}
 (\bar\rho_m^{\R})_*\mu_{\Gamma^m/\Z_2}&=\chi_m,\\
 (\bar\rho_m^{\C})_*\mu_{\Gamma^m_{\C}/U(1)}&=\chi_{2m}.
\end{align*}
\end{lemma}

\begin{proof}
The norm functions $x\mapsto\norm x_2$ and $z\mapsto\norm z_2$ are $1$-Lipschitz and invariant under the $\Z_2$-action and the $U(1)$-action, respectively. Applying \Cref{lem:quotient} to $(\R^m,\Z_2)$ and $(\C^m,U(1))$ proves the assertion.
\end{proof}

\section{Poincar\'e constants of the Hopf quotients}\label{sec:EKM}

This section transfers the finite-dimensional computation of $C_{2,2}$ to the two Hopf-quotient pyramids by box lower semicontinuity.

The standard extension of $C_{2,2}$ to pyramids is monotone and lower semicontinuous with respect to weak convergence of pyramids \cite{EKM}. Hence the following lemma is a special case of \cite[Lemma~3.2]{EKM}.

\begin{lemma}[Esaki--Kazukawa--Mitsuishi] \label{lem:C-closure}
Let $\{Y_n\}$ be a sequence of mm-spaces, and let $\mathcal P$ be a pyramid.
Suppose that $\mathcal P_{Y_n}\subset\mathcal P$ for every $n$ and that
$\{\mathcal P_{Y_n}\}$ converges weakly to $\mathcal P$ as $n\to\infty$.
If $C_{2,2}(Y_n)=C_0$ for every $n$ for a constant $C_0$, then
\[
 C_{2,2}(\mathcal P)=C_0.
\]
\end{lemma}

For $R>0$, let
\[
 \tau_R(t)\coloneqq\max\{-R,\min\{t,R\}\}
\]
be the \emph{truncation function}. It is $1$-Lipschitz.

\begin{proposition}\label{prop:finite-C}
For every $n\ge1$,
\begin{equation*}
 C_{2,2}(\Gamma^n/\Z_2)=\frac1{\sqrt2},
 \qquad
 C_{2,2}(\Gamma^n_{\C}/U(1))=\frac1{\sqrt2}.
\end{equation*}
\end{proposition}

\begin{proof}
Consider the real quotient. Take $\bar f\in\Lipb(\Gamma^n/\Z_2)$ and lift it to $f=\bar f\circ\pi_{\Z_2}$. By \Cref{lem:quotient}, $f$ is bounded, Lipschitz, and even. Moreover, $\pi_{\Z_2}(B_r(x))\subset B_r([x])$, and \Cref{lem:quotient} also gives $d_{\Gamma^n/\Z_2}(\pi_{\Z_2}(u),\pi_{\Z_2}(v))\le d_{\R^n}(u,v)$. Therefore,
\begin{equation*}
 \lipa(f)(x)\le\lipa(\bar f)([x]).
\end{equation*}
This inequality is sufficient for the Poincar\'e upper bound. By \Cref{lem:Lip-Sobolev}, $f\in W^{1,2}(\gamma^n)$. The function $f-\int f\dd\gamma^n$ is even and has zero mean, so part~(1) of \Cref{lem:gap-hierarchy} gives
\[
 \Var_{\gamma^n}(f)\le\frac12\int_{\R^n}\norm{\nabla f}_2^2\dd\gamma^n.
\]
Using also the almost-everywhere inequality $\norm{\nabla f}_2\le\lipa(f)$ and the definition of the quotient measure, we obtain
\begin{align*}
 \Var_{\Gamma^n/\Z_2}(\bar f)
 &=\Var_{\gamma^n}(f)\\
 &\le\frac12\int_{\R^n}\norm{\nabla f}_2^2\dd\gamma^n
 \le\frac12\int_{\R^n}\lipa(f)^2\dd\gamma^n\\
 &\le\frac12\int_{\Gamma^n/\Z_2}\lipa(\bar f)^2\dd\mu_{\Gamma^n/\Z_2}.
\end{align*}
Thus $C_{2,2}(\Gamma^n/\Z_2)\le1/\sqrt2$.

For the reverse inequality, set $h_R([x])\coloneqq\tau_R(\norm x_2^2-n)$. By \Cref{lem:radial-quotient}, $h_R\in\Lipb(\Gamma^n/\Z_2)$ and, almost everywhere,
\[
 \lipa(h_R)([x])
 \le2\norm x_2\,\ind_{\{\abs{\norm x_2^2-n}<R\}}.
\]
If $V\sim\gamma^n$, then \Cref{lem:chi} and the dominated convergence theorem give
\begin{align*}
 \Var_{\Gamma^n/\Z_2}(h_R)&\longrightarrow\Var(\norm V_2^2)=2n\quad(R\to\infty),\\
 \int\lipa(h_R)^2\dd\mu_{\Gamma^n/\Z_2}
 &\le\mathbb E\left[4\norm V_2^2
    \ind_{\{\abs{\norm V_2^2-n}<R\}}\right]
 \longrightarrow4n\quad(R\to\infty).
\end{align*}
Applying \eqref{eq:C22-def} to $h_R$ and letting $R\to\infty$ gives
$C_{2,2}(\Gamma^n/\Z_2)^2\ge1/2$.

The complex quotient is analogous. A $U(1)$-invariant lift is even because $-1\in U(1)$, so the preceding argument gives the upper bound $1/\sqrt2$. Under our normalization, $\Gamma^n_{\C}$ has $2n$ real standard normal coordinates, so the mean squared radius is $2n$. For the lower bound, set $h_R^{\C}([z])\coloneqq\tau_R(\norm z_2^2-2n)$ and let $Z\sim\gamma^{2n}$. The same estimate as above shows that, as $R\to\infty$, the variance of $h_R^{\C}$ converges to $\Var(\norm Z_2^2)=4n$, whereas the integral of $\lipa(h_R^{\C})^2$ is bounded above by a quantity converging to $\int4\norm Z_2^2\dd\gamma^{2n}=8n$. Thus \eqref{eq:C22-def} again gives $C_{2,2}(\Gamma^n_{\C}/U(1))^2\ge1/2$.
\end{proof}

\begin{proof}[Proof of \Cref{thm:C22}]
The associated pyramids generated by $\Gamma^n/\Z_2$ and
$\Gamma^n_{\C}/U(1)$ form increasing sequences, and the two quotient
pyramids are the weak limits of the corresponding increasing sequences.
By \Cref{prop:finite-C}, the constants along both sequences are
equal to $1/\sqrt2$.
Applying \Cref{lem:C-closure} to each sequence proves the theorem.
\end{proof}

\section{Separation of the Gaussian pyramid from the two quotients}

This section estimates the separation distances needed to distinguish the Gaussian pyramid from its two Hopf quotients.

\subsection{Separation distance of the Gaussian pyramid}

\begin{proposition} \label{prop:Sep-Gaussian}
For every $\alpha,\beta > 0$ with $\alpha+\beta\le1$,
\begin{equation*}
 \Sep(\Gamma^\infty;\alpha,\beta)
 =\Phi^{-1}(1-\beta)-\Phi^{-1}(\alpha).
\end{equation*}
In particular, for $0<\kappa\le1/2$,
\begin{equation*}
 \Sep(\Gamma^\infty;1/2,\kappa)=\Phi^{-1}(1-\kappa).
\end{equation*}
\end{proposition}

\begin{proof}
We first consider a finite-dimensional Gaussian space. Suppose that $\gamma^n(A)\ge\alpha$, $\gamma^n(B)\ge\beta$, and $d_{\R^n}(A,B)\ge r$. For every $0<r_0<r$, we have $A_{r_0}\cap B=\varnothing$, so \Cref{thm:gauss-iso} gives
\[
 \beta\le1-\gamma^n(A_{r_0})
 \le1-\Phi(\Phi^{-1}(\alpha)+r_0).
\]
Therefore,
\[
 r_0\le\Phi^{-1}(1-\beta)-\Phi^{-1}(\alpha).
\]
Letting $r_0\uparrow r$ gives the upper bound. Consider the parallel half-spaces
\[
 A=\{x_1\le\Phi^{-1}(\alpha)\},
 \qquad
 B=\{x_1\ge\Phi^{-1}(1-\beta)\}.
\]
Since the marginal distribution of the first coordinate is $N(0,1)$,
\[
 \gamma^n(A)=\Phi(\Phi^{-1}(\alpha))=\alpha,
 \qquad
 \gamma^n(B)=1-\Phi(\Phi^{-1}(1-\beta))=\beta.
\]
Moreover, $\alpha+\beta\le1$ implies $\Phi^{-1}(\alpha)\le\Phi^{-1}(1-\beta)$. For every $x\in A$ and $y\in B$,
\[
 \norm{x-y}_2\ge y_1-x_1
 \ge\Phi^{-1}(1-\beta)-\Phi^{-1}(\alpha),
\]
and equality holds by taking boundary points of the two half-spaces on the first coordinate axis. Therefore,
\[
 d_{\R^n}(A,B)=\Phi^{-1}(1-\beta)-\Phi^{-1}(\alpha),
\]
so these half-spaces attain equality in the upper bound.
Set
\[
 F(a,b)\coloneqq\Phi^{-1}(1-b)-\Phi^{-1}(a).
\]
For every $n\ge1$ and $a,b>0$ with $a+b\le1$, the preceding calculation and \cite[Proposition~4.4]{OzawaShioya} give
\begin{equation}
 \Sep(\mathcal P_{\Gamma^n};a,b)
 =\Sep(\Gamma^n;a,b)=F(a,b).                                    \label{eq:Sep-associated-Gaussian}
\end{equation}

The limit formula for the separation distance under weak convergence $\mathcal P_n\to\mathcal P$ of pyramids \cite[Theorem~4.9]{OzawaShioya} states that, for every $0 < \kappa_0,\kappa_1 \le 1$,
\[
 \Sep(\mathcal P;\kappa_0,\kappa_1)=\lim_{\varepsilon\to0^+}\liminf_{n\to\infty}\Sep(\mathcal P_n;\kappa_0-\varepsilon,\kappa_1-\varepsilon).
\]
Applying this formula with $\mathcal P_n=\mathcal P_{\Gamma^n}$, $\mathcal P=\Gamma^\infty$, $\kappa_0=\alpha$, and $\kappa_1=\beta$, we obtain
\[
 \Sep(\Gamma^\infty;\alpha,\beta)=\lim_{\varepsilon\to0^+}\liminf_{n\to\infty}\Sep(\mathcal P_{\Gamma^n};\alpha-\varepsilon,\beta-\varepsilon).
\]
For every $0<\varepsilon<\min\{\alpha,\beta\}$, equation \eqref{eq:Sep-associated-Gaussian} gives, for every $n\ge1$,
\[
 \Sep(\mathcal P_{\Gamma^n};\alpha-\varepsilon,\beta-\varepsilon)=F(\alpha-\varepsilon,\beta-\varepsilon).
\]
This expression does not depend on $n$. Taking first the lower limit as $n\to\infty$ and then letting $\varepsilon\downarrow0$, we obtain, by continuity of $\Phi^{-1}$ on $(0,1)$,
\[
 \Sep(\Gamma^\infty;\alpha,\beta)=\lim_{\varepsilon\to0^+}F(\alpha-\varepsilon,\beta-\varepsilon)=F(\alpha,\beta).
\]
This proves the first assertion. The second assertion is the case $\alpha=1/2$.
\end{proof}

The following is a standard consequence of the Mills-ratio bounds \cite{Gordon}.

\begin{lemma}\label{lem:quantile}
\begin{equation*}
 \lim_{\kappa\downarrow0}\frac{\Phi^{-1}(1-\kappa)^2}{2\log(1/\kappa)}=1.
\end{equation*}
\end{lemma}

\begin{corollary}\label{cor:Theta-Gauss}
$\ThetaPI(\Gamma^\infty)=2$.
\end{corollary}

\begin{proof}
By \cite{EKM} we know $C_{2,2}(\Gamma^\infty)^2=1$. Combining this with \Cref{prop:Sep-Gaussian} and \Cref{lem:quantile}, we obtain
\[
 \ThetaPI(\Gamma^\infty)
 =\lim_{\kappa\downarrow0}
 \frac{\Phi^{-1}(1-\kappa)^2}{\log(1/\kappa)}=2.
\]
\end{proof}

\subsection{Separation estimates for the Hopf quotients}

The radial map on each quotient is a measure-preserving $1$-Lipschitz observable, and the half-line carrying its image measure $\chi_m$ belongs to the corresponding pyramid. Choosing on this one-dimensional factor a lower interval of measure $1/2$ and an upper half-line of measure $\kappa$ gives a lower bound for the separation distance of the quotient.

For $m\ge1$, write
\[
 H_m\coloneqq([0,\infty),\abs{\cdot},\chi_m).
\]

\begin{lemma}\label{lem:H-membership}
For every $m\ge1$,
\[
 H_m\in\Gamma^\infty/\Z_2,
 \qquad
 H_{2m}\in\Gamma^\infty/U(1).
\]
\end{lemma}

\begin{proof}
By \Cref{lem:radial-quotient} we have
$H_m\prec\Gamma^m/\Z_2$ and $H_{2m}\prec\Gamma^m_{\C}/U(1)$,
which proves the lemma.
\end{proof}

For $0<\kappa\le1/2$, denote the medians of $\chi_1$ and $\chi_2$ by $m_1$ and $m_2$, respectively, and their \emph{upper $\kappa$-quantiles} by $q_1(\kappa)$ and $q_2(\kappa)$, respectively.

\begin{proposition}\label{prop:Theta-quotients}
\[
 \ThetaPI(\Gamma^\infty/\Z_2)=4,
 \qquad
 \ThetaPI(\Gamma^\infty/U(1))=4.
\]
\end{proposition}

\begin{proof}
We first prove the common upper bound.
The quotient maps are measure-preserving and $1$-Lipschitz, so
\[
 \Gamma^n/\Z_2\prec\Gamma^n,
 \qquad
 \Gamma^n_{\C}/U(1)\prec\Gamma^n_{\C}\simeq\Gamma^{2n}.
\]
Taking the box closures of the unions of the associated pyramids gives
\[
 \Gamma^\infty/\Z_2\subset\Gamma^\infty,
 \qquad
 \Gamma^\infty/U(1)\subset\Gamma^\infty.
\]
The separation distance is monotone under inclusion of pyramids by its definition.
Hence, for $G\in\{\Z_2,U(1)\}$ and $0<\kappa\le1/2$, \Cref{prop:Sep-Gaussian} yields
\[
 \Sep(\Gamma^\infty/G;1/2,\kappa)
 \le\Sep(\Gamma^\infty;1/2,\kappa)
 =\Phi^{-1}(1-\kappa).
\]
Using $C_{2,2}(\Gamma^\infty/G)^2=1/2$ from \Cref{thm:C22} and \Cref{lem:quantile}, we obtain
\[
 \ThetaPI(\Gamma^\infty/G)
 \le\lim_{\kappa\downarrow0}
 \frac{\Phi^{-1}(1-\kappa)^2}{(1/2)\log(1/\kappa)}
 =4.
\]

It remains to prove the lower bounds.
We first use $H_1\in\Gamma^\infty/\Z_2$. The distribution function of $S_1=\abs\xi$, where $\xi\sim N(0,1)$, is given by \eqref{eq:half-normal-cdf}. Its median and upper $\kappa$-quantile are
\[
 m_1=\Phi^{-1}(3/4),
 \qquad
 q_1(\kappa)=\Phi^{-1}(1-\kappa/2).
\]
The lower interval $[0,m_1]$ has $\chi_1$-measure $1/2$, and the upper half-line $[q_1(\kappa),\infty)$ has $\chi_1$-measure $\kappa$. Moreover, $\kappa\le1/2$ gives $q_1(\kappa)\ge m_1$, so the distance between these two sets is $q_1(\kappa)-m_1$. Since $H_1\in\Gamma^\infty/\Z_2$,
\begin{equation*}
 \Sep(\Gamma^\infty/\Z_2;1/2,\kappa)
 \ge\Sep(H_1;1/2,\kappa)
 \ge q_1(\kappa)-m_1.
\end{equation*}
Applying \Cref{lem:quantile} with $\kappa/2$, we obtain, as $\kappa\downarrow0$,
\[
 \frac{q_1(\kappa)^2}{2\log(2/\kappa)}\xrightarrow[\kappa\downarrow0]{}1,
 \qquad
 \frac{\log(2/\kappa)}{\log(1/\kappa)}
 =1+\frac{\log2}{\log(1/\kappa)}\xrightarrow[\kappa\downarrow0]{}1.
\]
Also, $q_1(\kappa)\xrightarrow[\kappa\downarrow0]{}\infty$ and $m_1$ is fixed, so
\[
 \frac{(q_1(\kappa)-m_1)^2}{q_1(\kappa)^2}
 =\left(1-\frac{m_1}{q_1(\kappa)}\right)^2\xrightarrow[\kappa\downarrow0]{}1.
\]
Using these three limits to rewrite the left-hand side gives
\[
 \frac{(q_1(\kappa)-m_1)^2}{\log(1/\kappa)}
 =2\cdot\frac{q_1(\kappa)^2}{2\log(2/\kappa)}
 \cdot\frac{\log(2/\kappa)}{\log(1/\kappa)}
 \cdot\left(1-\frac{m_1}{q_1(\kappa)}\right)^2
 \xrightarrow[\kappa\downarrow0]{}2\cdot1\cdot1\cdot1=2.
\]
Furthermore, \Cref{thm:C22} gives $C_{2,2}(\Gamma^\infty/\Z_2)^2=1/2$, and hence
\[
 \ThetaPI(\Gamma^\infty/\Z_2)
 \ge\lim_{\kappa\downarrow0}
 \frac{(q_1(\kappa)-m_1)^2}{(1/2)\log(1/\kappa)}=4.
\]

We next use $H_2\in\Gamma^\infty/U(1)$. Equation~\eqref{eq:rayleigh-cdf} shows that the upper-tail probability of $S_2\sim\chi_2$ is
\[
 \mathbb P(S_2\ge r)=e^{-r^2/2}\qquad(r\ge0).
\]
Therefore, its median and upper $\kappa$-quantile are
\[
 m_2=\sqrt{2\log2},
 \qquad
 q_2(\kappa)=\sqrt{2\log(1/\kappa)}.
\]
The lower interval $[0,m_2]$ has $\chi_2$-measure $1/2$, and the upper half-line $[q_2(\kappa),\infty)$ has $\chi_2$-measure $\kappa$. Moreover, $\kappa\le1/2$ gives $q_2(\kappa)\ge m_2$, so the distance between these two sets is $q_2(\kappa)-m_2$. Since $H_2\in\Gamma^\infty/U(1)$,
\begin{equation*}
 \Sep(\Gamma^\infty/U(1);1/2,\kappa)
 \ge\Sep(H_2;1/2,\kappa)
 \ge q_2(\kappa)-m_2.
\end{equation*}
Since $q_2(\kappa)^2=2\log(1/\kappa)$, while $q_2(\kappa)\to\infty$ as $\kappa\downarrow0$ and $m_2$ is fixed,
\[
 \frac{(q_2(\kappa)-m_2)^2}{\log(1/\kappa)}
 =2\left(1-\frac{m_2}{q_2(\kappa)}\right)^2
 \xrightarrow[\kappa\downarrow0]{}2.
\]
Furthermore, $C_{2,2}(\Gamma^\infty/U(1))^2=1/2$, and hence
\[
 \ThetaPI(\Gamma^\infty/U(1))
 \ge\lim_{\kappa\downarrow0}
 \frac{(q_2(\kappa)-m_2)^2}{(1/2)\log(1/\kappa)}=4.
\]
\end{proof}

\begin{proof}[Proof of \Cref{thm:intro-scale-separation}]
The Gaussian value follows from \Cref{cor:Theta-Gauss}, and the two quotient values follow from \Cref{prop:Theta-quotients}.
\end{proof}

\begin{corollary}\label{cor:Gauss-vs-quot}
For every $a,b>0$ and $G\in\{\Z_2,U(1)\}$,
\[
 a\Gamma^\infty\ne b(\Gamma^\infty/G).
\]
\end{corollary}

\begin{proof}
By the scale invariance of $\ThetaPI$ stated in the introduction, equality of $a\Gamma^\infty$ and $b(\Gamma^\infty/G)$ would imply $\ThetaPI(\Gamma^\infty)=\ThetaPI(\Gamma^\infty/G)$. However, \Cref{cor:Theta-Gauss,prop:Theta-quotients} show that the left-hand side is $2$ and the right-hand side is $4$, a contradiction.
\end{proof}

\section{Separation of the \texorpdfstring{$\Z_2$}{Z2} quotient from the \texorpdfstring{$U(1)$}{U(1)} quotient}

It remains to compare the two quotients with each other. We prove the strict containment stated in \Cref{thm:intro-strict-containment} by deriving the inclusion of pyramids from finite-dimensional quotient maps and showing that the half-normal line belongs to $\Gamma^\infty/\Z_2$ but not to $\Gamma^\infty/U(1)$.

\subsection{Inclusion of pyramids induced by quotient maps}

In what follows, $\Gamma^n_{\C}/\Z_2$ denotes the quotient by the action $z\mapsto\pm z$ of the subgroup $\{\pm1\}\subset U(1)$. This subgroup relation gives, for every $n$, the map
\[
 Q_n\colon\Gamma^n_{\C}/\Z_2\longrightarrow\Gamma^n_{\C}/U(1),
 \qquad
 [z]_{\Z_2}\longmapsto[z]_{U(1)}.
\]
If $[z]_{\Z_2}=[w]_{\Z_2}$, then $w=\varepsilon z$ for some $\varepsilon\in\{\pm1\}$. Since $\varepsilon\in U(1)$, we have $[z]_{U(1)}=[w]_{U(1)}$. Thus $Q_n$ is independent of the representative and is clearly surjective. Moreover,
\[
 d_{\C^n/U(1)}(Q_n([z]),Q_n([w]))
 =\inf_{\theta\in\R}\norm{z-e^{i\theta}w}_2
 \le\min_{\varepsilon\in\{\pm1\}}\norm{z-\varepsilon w}_2
 =d_{\C^n/\Z_2}([z],[w]).
\]
Therefore, $Q_n$ is $1$-Lipschitz and, in particular, Borel measurable. Let $\pi_{\Z_2}$ and $\pi_{U(1)}$ be the quotient maps from $\C^n$ to the two quotients. Since
\[
 \pi_{U(1)}=Q_n\circ\pi_{\Z_2},
\]
we have
\[
 (Q_n)_*\mu_{\Gamma^n_{\C}/\Z_2}
 =(Q_n)_*(\pi_{\Z_2})_*\gamma^{2n}
 =(\pi_{U(1)})_*\gamma^{2n}
 =\mu_{\Gamma^n_{\C}/U(1)}.
\]
Thus $Q_n$ is measure-preserving, and
\begin{equation*}
 \Gamma^n_{\C}/U(1)\prec\Gamma^n_{\C}/\Z_2
 \simeq\Gamma^{2n}/\Z_2.
\end{equation*}
This domination and the monotonicity of associated pyramids give
\[
 \bigcup_{n\ge1}\mathcal P_{\Gamma^n_{\C}/U(1)}
 \subset\Gamma^\infty/\Z_2.
\]
Since the right-hand side is closed with respect to the box distance, taking the box closure of the left-hand side gives
\begin{equation}
 \Gamma^\infty/U(1)\subset\Gamma^\infty/\Z_2.                   \label{eq:natural-inclusion}
\end{equation}
By \Cref{lem:H-membership}, $H_1\in\Gamma^\infty/\Z_2$. It therefore suffices to prove $H_1\notin\Gamma^\infty/U(1)$ in order to show that this inclusion is strict.

\subsection{The half-normal line does not belong to the \texorpdfstring{$U(1)$}{U(1)} quotient}\label{sec:H1-not-U1}

We use the following approximation and lifting result to transfer an observable on a box limit to finite-dimensional quotients.

\begin{lemma}\label{lem:observable-approx}
Suppose that $X_j$ converges to $X$ in the box distance, and let $f\in\Lip_1(X)$. Then there exist $f_j\in\Lip_1(X_j)$ such that
\begin{equation*}
 (f_j)_*\mu_{X_j}\Rightarrow f_*\mu_X
\end{equation*}
as $j\to\infty$. Moreover, if $f\ge0$, then each $f_j$ can be chosen to satisfy $f_j\ge0$.

Suppose also that $X_j\prec Y_j$ and that a \emph{dominating map} $p_j\colon Y_j\to X_j$ is given. Then
\[
 F_j\coloneqq f_j\circ p_j\in\Lip_1(Y_j)
\]
satisfies
\[
 (F_j)_*\mu_{Y_j}=(f_j)_*\mu_{X_j}\Rightarrow f_*\mu_X
\]
as $j\to\infty$.
\end{lemma}

\begin{proof}
Under box convergence, the set of all pushforward measures under $1$-Lipschitz maps to $\R$, namely, the \emph{$1$-measurement}, converges in the Hausdorff distance induced by the Prokhorov distance \cite[Proposition~5.5 and Lemma~5.12]{ShioyaBook}. Since $f_*\mu_X$ belongs to the $1$-measurement of $X$, we may choose $\tilde f_j\in\Lip_1(X_j)$ such that
\[
 (\tilde f_j)_*\mu_{X_j}\Rightarrow f_*\mu_X
\]
as $j\to\infty$. If $f\ge0$, set $f_j\coloneqq\max\{\tilde f_j,0\}$, and otherwise set $f_j\coloneqq\tilde f_j$. The positive-part map is continuous, and $\max\{f,0\}=f$. Thus the asserted convergence holds, and in the case $f\ge0$ we have $f_j\ge0$. Finally, because $p_j$ is measure-preserving, $F_j=f_j\circ p_j$ satisfies
\[
 (F_j)_*\mu_{Y_j}
 =(f_j)_*(p_j)_*\mu_{Y_j}
 =(f_j)_*\mu_{X_j}
 \Rightarrow f_*\mu_X
\]
as $j\to\infty$.
\end{proof}

The comparison of third moments requires a dimension-independent bound on the $L^4$-norm of the second Wiener chaos component. \emph{Hypercontractivity}, due to Nelson \cite{nelson1973}, gives the following degree-two estimate; see \cite[(1.71), p.~63]{Nualart}.

\begin{lemma}\label{lem:H2-hypercontractivity}
For every $q\in\mathcal H_2=\mathcal Q_n$,
\begin{equation*}
 \norm{q}_4\le3\norm{q}_2.
\end{equation*}
\end{lemma}

The following orthogonal equivariance shows that $P_2$ preserves $U(1)$-invariance. Let
\[
O(n)\coloneqq\{O\in\R^{n\times n}\mid O^{\mathsf T}O=I_n\}
\]
be the \emph{real orthogonal group}.

\begin{lemma}\label{lem:P2-equiv}
For $O\in O(n)$, define $\mathcal U_O\colon L^2(\gamma^n)\to L^2(\gamma^n)$ by $\mathcal U_Ou(x)\coloneqq u(O^{-1}x)$. Then
\[
 P_2\mathcal U_O=\mathcal U_OP_2.
\]
Consequently, if $K\subset O(n)$ is a subgroup and $u\in L^2(\gamma^n)$ satisfies $u(O^{-1}x)=u(x)$ for every $O\in K$, then $P_2u(O^{-1}x)=P_2u(x)$ also holds for every $O\in K$.
\end{lemma}

\begin{proof}
Since $O$ preserves Gaussian measure, a change of variables gives, for every $u,v\in L^2(\gamma^n)$,
\[
 \inner{\mathcal U_Ou}{\mathcal U_Ov}_{L^2(\gamma^n)}
 =\inner{u}{v}_{L^2(\gamma^n)}.
\]
For $A\in\operatorname{Sym}(n)$, using $O^{-1}=O^{\mathsf T}$ and
\[
 \tr(OAO^{\mathsf T})=\tr(AO^{\mathsf T}O)=\tr A,
\]
we obtain
\[
 (\mathcal U_Oq_A)(x)
 =q_A(O^{\mathsf T}x)
 =x^{\mathsf T}OAO^{\mathsf T}x-\tr A
 =q_{OAO^{\mathsf T}}(x).
\]
Thus $\mathcal U_O$ preserves $\mathcal Q_n$. For every $q\in\mathcal Q_n$, we have $\mathcal U_{O^{-1}}q\in\mathcal Q_n$, and hence
\[
 \inner{\mathcal U_O(u-P_2u)}{q}_{L^2(\gamma^n)}
 =\inner{u-P_2u}{\mathcal U_{O^{-1}}q}_{L^2(\gamma^n)}=0.
\]
Therefore, $\mathcal U_OP_2u$ is the element of $\mathcal Q_n$ closest to $\mathcal U_Ou$. Its uniqueness gives $P_2\mathcal U_Ou=\mathcal U_OP_2u$. The final assertion follows by applying this identity to each $O\in K$.
\end{proof}

For a centered quadratic form that retains $U(1)$-invariance, we use the following upper bound for the third moment.

\begin{lemma} \label{lem:U1-H2}
On $\C^N\simeq\R^{2N}$, set
\[
 U_\theta\colon\C^N\to\C^N,
 \qquad U_\theta z\coloneqq e^{i\theta}z
 \qquad(\theta\in\R).
\]
Suppose that $q\in\mathcal Q_{2N}$ satisfies
\begin{equation*}
 q(U_\theta z)=q(z)
 \qquad(z\in\C^N,\ \theta\in\R).
\end{equation*}
Then there exist complex orthonormal coordinates $z=(z_1,\dots,z_N)$ on $\C^N$ and real numbers $\lambda_1,\dots,\lambda_N$ such that
\begin{equation*}
 q(z)=\sum_{k=1}^N\lambda_k(\abs{z_k}^2-2).
\end{equation*}
Moreover,
\begin{equation*}
 \int_{\R^{2N}}q^2\dd\gamma^{2N}=4\sum_{k=1}^N\lambda_k^2,
 \qquad
 \int_{\R^{2N}}q^3\dd\gamma^{2N}=16\sum_{k=1}^N\lambda_k^3,
\end{equation*}
and consequently
\begin{equation*}
 \left|\int_{\R^{2N}}q^3\dd\gamma^{2N}\right|
 \le2\left(\int_{\R^{2N}}q^2\dd\gamma^{2N}\right)^{3/2}.
\end{equation*}
\end{lemma}

\begin{proof}
We first show that $U(1)$-invariance forces the representing matrix $A$ to commute with the complex structure $J$. This makes $A$ Hermitian with respect to the complex inner product. Unitary diagonalization then gives real eigenvalues in $2\times2$ real blocks and yields the asserted representation. We finally compute the moments directly.

By \Cref{lem:quad-trace}, there exists a unique $A\in\operatorname{Sym}(2N)$ such that
\begin{equation*}
 q(x)=q_A(x)=x^{\mathsf T}Ax-\tr A.
\end{equation*}
Define the \emph{complex structure} on $\C^N\simeq\R^{2N}$ by
\[
 J(x_1,y_1,\dots,x_N,y_N)
 \coloneqq(-y_1,x_1,\dots,-y_N,x_N).
\]
Then $U_\theta=e^{\theta J}$. The assumed invariance implies
\[
 x^{\mathsf T}(U_\theta^{\mathsf T}AU_\theta-A)x=0
 \qquad(x\in\R^{2N}).
\]
The matrix $B_\theta\coloneqq U_\theta^{\mathsf T}AU_\theta-A$ is real and symmetric. For every $x,y\in\R^{2N}$, the polarization identity gives
\[
 4x^{\mathsf T}B_\theta y
 =(x+y)^{\mathsf T}B_\theta(x+y)
  -(x-y)^{\mathsf T}B_\theta(x-y)=0.
\]
Thus $B_\theta=0$, that is, $U_\theta^{\mathsf T}AU_\theta=A$. Since $U_\theta=e^{\theta J}$ and $J^{\mathsf T}=-J$,
\[
 U_\theta^{\mathsf T}=e^{\theta J^{\mathsf T}}=e^{-\theta J},
 \qquad
 \left.\frac{\dd}{\dd\theta}e^{\theta J}\right|_{\theta=0}=J,
 \qquad
 \left.\frac{\dd}{\dd\theta}e^{-\theta J}\right|_{\theta=0}=-J.
\]
Differentiating $U_\theta^{\mathsf T}AU_\theta=A$ at $\theta=0$ and using the product rule, we obtain
\[
 \left.\frac{\dd}{\dd\theta}
 \bigl(e^{-\theta J}Ae^{\theta J}\bigr)\right|_{\theta=0}
 =-JA+AJ=0.
\]
Therefore,
\[
 AJ=JA.
\]
Applying this equality to $w\in\C^N\simeq\R^{2N}$ gives $AJw=JAw$. Since $Jw$ represents $iw$,
\[
 A(iw)=A(Jw)=JAw=iAw.
\]
Hence $A$ can be regarded as a complex-linear operator on $\C^N$.

Write $\langle\cdot,\cdot\rangle_{\C}$ for the \emph{complex inner product} and $\langle\cdot,\cdot\rangle_{\R}$ for the \emph{real Euclidean inner product}, and take the complex inner product to be complex linear in the first variable. The real symmetry of $A$ gives
\[
 \Re\langle Az,w\rangle_{\C}
 =\langle Az,w\rangle_{\R}
 =\langle z,Aw\rangle_{\R}
 =\Re\langle z,Aw\rangle_{\C}.
\]
Replace $w$ by $iw$. Conjugate linearity in the second variable and the identity $A(iw)=iAw$ give
\begin{align*}
 \Re\langle Az,iw\rangle_{\C}
 &=\Re\bigl(-i\langle Az,w\rangle_{\C}\bigr)
 =\Im\langle Az,w\rangle_{\C},\\
 \Re\langle z,A(iw)\rangle_{\C}
 &=\Re\bigl(-i\langle z,Aw\rangle_{\C}\bigr)
 =\Im\langle z,Aw\rangle_{\C}.
\end{align*}
The imaginary parts are therefore also equal, so $\langle Az,w\rangle_{\C}=\langle z,Aw\rangle_{\C}$. Thus $A$ is Hermitian with respect to the complex inner product. By the spectral theorem for Hermitian matrices, after changing complex orthonormal coordinates, the real representation of $A$ is
\[
 \operatorname{diag}(\lambda_1I_2,\dots,\lambda_NI_2).
\]
Writing $z_k=x_k+iy_k$ in these coordinates, we have
\[
 x^{\mathsf T}Ax
 =\sum_{k=1}^N\lambda_k(x_k^2+y_k^2)
 =\sum_{k=1}^N\lambda_k|z_k|^2,
 \qquad
 \tr A=2\sum_{k=1}^N\lambda_k.
\]
Substituting these identities into the definition of $q_A$ gives
\[
 q(z)=\sum_{k=1}^N\lambda_k|z_k|^2-2\sum_{k=1}^N\lambda_k
 =\sum_{k=1}^N\lambda_k(|z_k|^2-2),
\]
which proves the asserted representation.

This complex unitary change of coordinates preserves the real Euclidean norm, so it is also an orthogonal transformation of $\R^{2N}$ and preserves the standard Gaussian measure $\gamma^{2N}$. We may therefore compute the following moments in the new coordinates. Let $Z=(Z_1,\dots,Z_N)\sim\gamma^{2N}$ and write $Z_k=\xi_k+i\eta_k$, where $\xi_1,\eta_1,\dots,\xi_N,\eta_N$ are independent standard normal variables. Set
\[
 W_k\coloneqq|Z_k|^2-2=\xi_k^2+\eta_k^2-2.
\]
By \Cref{lem:normal-mom}, $\mathbb E\xi_k^2=\mathbb E\eta_k^2=1$, $\mathbb E\xi_k^4=\mathbb E\eta_k^4=3$, and $\mathbb E\xi_k^6=\mathbb E\eta_k^6=15$. The independence of $\xi_k$ and $\eta_k$ gives
\begin{align*}
 \mathbb E(\xi_k^2+\eta_k^2)^2
 &=\mathbb E\xi_k^4+2\mathbb E\xi_k^2\mathbb E\eta_k^2+\mathbb E\eta_k^4 = 8,\\
 \mathbb E(\xi_k^2+\eta_k^2)^3
 &=\mathbb E\xi_k^6
   +3\mathbb E\xi_k^4\mathbb E\eta_k^2
   +3\mathbb E\xi_k^2\mathbb E\eta_k^4
   +\mathbb E\eta_k^6 = 48,\\
 \mathbb E(\xi_k^2+\eta_k^2) &= 2.
\end{align*}
Therefore,
\begin{align*}
 \mathbb EW_k
 &=\mathbb E(\xi_k^2+\eta_k^2)-2=0,\\
 \mathbb EW_k^2
 &=\mathbb E(\xi_k^2+\eta_k^2)^2
   -4\mathbb E(\xi_k^2+\eta_k^2)+4
 =4,\\
 \mathbb EW_k^3
 &=\mathbb E(\xi_k^2+\eta_k^2)^3
   -6\mathbb E(\xi_k^2+\eta_k^2)^2
   +12\mathbb E(\xi_k^2+\eta_k^2)-8 =16.
\end{align*}
The asserted representation gives $q(Z)=\sum_{k=1}^N\lambda_kW_k$. The variables $W_k$ for distinct $k$ are independent and centered. Thus, if $k\ne\ell$,
\[
 \mathbb E(W_kW_\ell)=\mathbb EW_k\,\mathbb EW_\ell=0,
 \qquad
 \mathbb E(W_k^2W_\ell)=\mathbb EW_k^2\,\mathbb EW_\ell=0,
\]
and $\mathbb E(W_kW_\ell W_m)=0$ for distinct $k,\ell,m$. Therefore,
\begin{align*}
 \mathbb Eq(Z)^2
 &=\sum_{k=1}^N\lambda_k^2\mathbb EW_k^2
   +2\sum_{1\le k<\ell\le N}\lambda_k\lambda_\ell\mathbb E(W_kW_\ell)
 =4\sum_{k=1}^N\lambda_k^2,\\
 \mathbb Eq(Z)^3
 &=\sum_{k=1}^N\lambda_k^3\mathbb EW_k^3
   +3\sum_{\substack{1\le k,\ell\le N\\k\ne\ell}}
      \lambda_k^2\lambda_\ell\mathbb E(W_k^2W_\ell)\\
 &\quad
   +6\sum_{1\le k<\ell<m\le N}
      \lambda_k\lambda_\ell\lambda_m\mathbb E(W_kW_\ell W_m)\\
 &=16\sum_{k=1}^N\lambda_k^3.
\end{align*}
This proves the moment identities. Finally,
\begin{align*}
 \left| \int_{\R^{2N}}q^3\dd\gamma^{2N} \right|
 &=16 \left| \sum_{k=1}^N\lambda_k^3 \right|
 \le16\sum_{k=1}^N|\lambda_k|^3
 \le16\left(\sum_{k=1}^N\lambda_k^2\right)^{3/2}\\
 &=2\left(4\sum_{k=1}^N\lambda_k^2\right)^{3/2}
 =2\left(\int_{\R^{2N}}q^2\dd\gamma^{2N}\right)^{3/2}.
\end{align*}
This proves the bound.
\end{proof}

\begin{theorem} \label{thm:H1-not-U1}
\[
 H_1\notin\Gamma^\infty/U(1).
\]
Consequently, the inclusion \eqref{eq:natural-inclusion} is strict.
\end{theorem}

\begin{proof}
The proof derives a contradiction from the third moments. The centered squares of the lifted observables are approximated in $L^3$ by $U(1)$-invariant centered quadratic forms with limiting second and third moments $2$ and $8$, whereas \Cref{lem:U1-H2} bounds the limiting third moment by $4\sqrt2$.

Suppose, to the contrary, that $H_1\in\Gamma^\infty/U(1)$. Since $\Gamma^\infty/U(1)$ is the box closure of the union of the associated pyramids generated by the finite-dimensional quotients, there exist positive integers $N_j$ and mm-spaces $X_j$ such that
\[
 X_j\prec\Gamma^{N_j}_{\C}/U(1),
 \qquad
 X_j\longrightarrow H_1\quad\text{in the box distance as }j\to\infty.
\]
Let $p_j\colon\Gamma^{N_j}_{\C}/U(1)\to X_j$
be a measure-preserving $1$-Lipschitz map realizing the domination. The identity function $r(t)=t$ on $H_1$ is nonnegative and $1$-Lipschitz and satisfies $r_*\chi_1=\chi_1$. By \Cref{lem:observable-approx}, choose nonnegative observables $f_j\in\Lip_1(X_j)$ approximating $r$ such that $(f_j)_*\mu_{X_j}\Rightarrow\chi_1$ as $j\to\infty$. Using the quotient map $\pi_j\colon\C^{N_j}\to\C^{N_j}/U(1)$, set
\begin{equation*}
 r_j\coloneqq f_j\circ p_j\circ\pi_j\colon\C^{N_j}\longrightarrow[0,\infty).
\end{equation*}
Then $r_j$ is $U(1)$-invariant and $1$-Lipschitz, and
\begin{equation}
 (r_j)_*\gamma^{2N_j}\Rightarrow\chi_1
 \qquad(j\to\infty).                                            \label{eq:rj-weak}
\end{equation}

In what follows, expectations and $L^p$-norms are taken with respect to $\gamma^{2N_j}$. If $\xi\sim N(0,1)$, then the limiting distribution is that of $\abs\xi$. By \Cref{lem:normal-mom},
\[
 \mathbb E\abs\xi^2=1,
 \qquad
 \mathbb E\abs\xi^4=3,
 \qquad
 \mathbb E\abs\xi^6=15.
\]
Therefore,
\[
 \Var(\abs\xi^2)=2,
 \qquad
 \mathbb E(\abs\xi^2-1)^3=8.
\]
Applying \eqref{eq:moment-conv} in \Cref{lem:UI} to \eqref{eq:rj-weak} with $\psi(t)=t^2,t^4,t^6$, respectively, gives $\mathbb Er_j^2\to1$, $\mathbb Er_j^4\to3$, and $\mathbb Er_j^6\to15$ as $j\to\infty$. Hence
\begin{align}
 \mathbb Er_j^2&\longrightarrow1,                                 \label{eq:rj2}\\
 \Var(r_j^2)&=\mathbb Er_j^4-(\mathbb Er_j^2)^2\longrightarrow2,  \label{eq:rj4}\\
 \mathbb E(r_j^2-\mathbb Er_j^2)^3
 &=\mathbb Er_j^6-3(\mathbb Er_j^2)(\mathbb Er_j^4)
   +2(\mathbb Er_j^2)^3\longrightarrow8                         \label{eq:rj6}
\end{align}
as $j\to\infty$.

\medskip
\noindent\textbf{Step 1: The centered square concentrates near a quadratic form.}
Set $h_j\coloneqq r_j^2-\mathbb Er_j^2$. Then $h_j$ has zero mean and is $U(1)$-invariant. In particular, it is even because $-1\in U(1)$. Using $P_2$ defined by \eqref{eq:chaos-decomposition}, set
\begin{equation}
 q_j\coloneqq P_2h_j,
 \qquad
 s_j\coloneqq h_j-q_j.                                           \label{eq:hj-quadratic-split}
\end{equation}
Since $q_j$ is a quadratic polynomial, it belongs to $W^{1,2}(\gamma^{2N_j})$. By \Cref{lem:Lip-Sobolev}, $h_j\in W^{1,2}(\gamma^{2N_j})$, so $s_j\in W^{1,2}(\gamma^{2N_j})$ as well. Moreover, $s_j$ is even, has zero mean, and is orthogonal to $\mathcal Q_{2N_j}$. Set
\begin{equation*}
 \mathcal E_j(u)\coloneqq\int\norm{\nabla u}_2^2\dd\gamma^{2N_j}
 \qquad\bigl(u\in W^{1,2}(\gamma^{2N_j})\bigr).
\end{equation*}
By \Cref{lem:Lip-Sobolev} and the almost-everywhere inequality $\norm{\nabla r_j}_2\le1$,
\begin{equation}
 \mathcal E_j(h_j)
 =4\int r_j^2\norm{\nabla r_j}_2^2\dd\gamma^{2N_j}
 \le4\mathbb Er_j^2.                                              \label{eq:hj-energy-upper}
\end{equation}
Applying \eqref{eq:chaos-energy} to $q_j=P_2h_j$, whose only chaos component has degree $2$, gives
\begin{equation}
 \mathcal E_j(q_j)=2\norm{q_j}_2^2.                              \label{eq:qj-energy}
\end{equation}
Also, $P_2s_j=P_2h_j-P_2^2h_j=0$, and $P_mq_j=0$ for $m\ne2$. Applying \eqref{eq:chaos-energy-polarized} to $q_j$ and $s_j$, we therefore obtain
\begin{equation}
 \int\inner{\nabla q_j}{\nabla s_j}\dd\gamma^{2N_j}
 =\sum_{m=1}^\infty m\inner{P_mq_j}{P_ms_j}_{L^2(\gamma^{2N_j})}
 =2\inner{q_j}{P_2s_j}=0.                                       \label{eq:qj-energy-orth}
\end{equation}
Using $h_j=q_j+s_j$ from \eqref{eq:hj-quadratic-split} and the linearity of the gradient, we have
\begin{align*}
 \mathcal E_j(h_j)
 &=\int\norm{\nabla q_j+\nabla s_j}_2^2\dd\gamma^{2N_j}\\
 &=\mathcal E_j(q_j)
   +2\int\inner{\nabla q_j}{\nabla s_j}\dd\gamma^{2N_j}
   +\mathcal E_j(s_j)\\
 &=2\norm{q_j}_2^2+\mathcal E_j(s_j),
\end{align*}
where the last equality uses \eqref{eq:qj-energy} and \eqref{eq:qj-energy-orth}. The $L^2$-orthogonality also gives
\[
 \norm{h_j}_2^2=\norm{q_j}_2^2+\norm{s_j}_2^2.
\]
Subtracting twice the second identity from the first, we obtain
\[
 \mathcal E_j(h_j)-2\norm{h_j}_2^2
 =\mathcal E_j(s_j)-2\norm{s_j}_2^2.
\]
Applying part~(2) of \Cref{lem:gap-hierarchy} to $s_j$ and then using \eqref{eq:rj2}, \eqref{eq:rj4}, and \eqref{eq:hj-energy-upper}, we obtain
\begin{align}
 0&\le2\norm{s_j}_2^2
 \le\mathcal E_j(h_j)-2\norm{h_j}_2^2\\
 &\le4\mathbb Er_j^2-2\Var(r_j^2)\longrightarrow0
 \qquad(j\to\infty).                                           \label{eq:spectral-defect}
\end{align}
Equations \eqref{eq:hj-quadratic-split} and \eqref{eq:spectral-defect} give
\begin{equation}
 \norm{h_j-q_j}_2^2
 =\norm{s_j}_2^2
 \le\frac12\bigl(\mathcal E_j(h_j)-2\norm{h_j}_2^2\bigr)
 \longrightarrow0\qquad(j\to\infty).                           \label{eq:L2-q}
\end{equation}

We strengthen this to $L^3$ convergence in order to transfer convergence of the third moments from $h_j$ to $q_j$. By \Cref{lem:UI}, $\sup_j\mathbb Er_j^8<\infty$, and \eqref{eq:rj2} also gives $\sup_j\mathbb Er_j^2<\infty$. Since $h_j=r_j^2-\mathbb Er_j^2$, the triangle inequality gives
\[
 \norm{h_j}_4
 \le\norm{r_j^2}_4+\mathbb Er_j^2
 =(\mathbb Er_j^8)^{1/4}+\mathbb Er_j^2.
\]
Therefore,
\begin{equation}
 \sup_j\norm{h_j}_4<\infty.                                      \label{eq:hj-L4}
\end{equation}
We have $q_j=P_2h_j\in\mathcal H_2$, and \eqref{eq:chaos-Parseval} gives $\norm{q_j}_2\le\norm{h_j}_2$. The right-hand side is uniformly bounded by \eqref{eq:rj4}. Thus \Cref{lem:H2-hypercontractivity} gives
\begin{equation}
 \sup_j\norm{q_j}_4<\infty.                                     \label{eq:qj-L4}
\end{equation}
Since $s_j=h_j-q_j$, equations \eqref{eq:hj-L4} and \eqref{eq:qj-L4} give $\sup_j\norm{s_j}_4<\infty$. The Cauchy--Schwarz inequality yields
\[
 \norm{s_j}_3^3
 =\int |s_j|\,|s_j|^2\dd\gamma^{2N_j}
 \le\norm{s_j}_2\norm{s_j}_4^2.
\]
It follows from \eqref{eq:L2-q} that
\begin{equation}
 \norm{s_j}_3
 \le\norm{s_j}_2^{1/3}\norm{s_j}_4^{2/3}
 \longrightarrow0\qquad(j\to\infty).                           \label{eq:L3-q}
\end{equation}
Equation~\eqref{eq:rj4} gives $\norm{h_j}_2^2=\Var(r_j^2)\to2$, and \eqref{eq:L2-q} gives $\norm{q_j-h_j}_2\to0$, as $j\to\infty$. Therefore, the reverse triangle inequality gives
\[
 \abs{\norm{q_j}_2-\norm{h_j}_2}
 \le\norm{q_j-h_j}_2\longrightarrow0.
\]
Thus $\norm{q_j}_2\to\sqrt2$ as $j\to\infty$, or equivalently, $\mathbb Eq_j^2\to2$. Moreover,
\[
 \abs{q_j^3-h_j^3}
 \le\abs{q_j-h_j}(\abs{q_j}^2+\abs{q_jh_j}+\abs{h_j}^2),
\]
and H\"older's inequality gives
\[
 \abs{\mathbb E(q_j^3-h_j^3)}
 \le\norm{q_j-h_j}_3
 \bigl(\norm{q_j}_3^2+\norm{q_j}_3\norm{h_j}_3+\norm{h_j}_3^2\bigr).
\]
By \eqref{eq:hj-L4}, \eqref{eq:qj-L4}, and $\norm{g}_3\le\norm{g}_4$, the expression in parentheses is uniformly bounded. Equation~\eqref{eq:L3-q} gives $\norm{q_j-h_j}_3=\norm{s_j}_3\to0$ as $j\to\infty$. Hence $\abs{\mathbb Eq_j^3-\mathbb Eh_j^3}\to0$ as $j\to\infty$. Combining these facts with \eqref{eq:rj4} and \eqref{eq:rj6}, we obtain
\begin{equation}
 \mathbb Eq_j^2\longrightarrow2,
 \qquad
 \mathbb Eq_j^3\longrightarrow8
 \qquad(j\to\infty).                                             \label{eq:qj-mom-limit}
\end{equation}

\medskip
\noindent\textbf{Step 2: The third-moment bound for $U(1)$-invariant centered quadratic forms.}
By \eqref{eq:hj-quadratic-split}, $q_j=P_2h_j$, and $h_j$ is $U(1)$-invariant. For every $\theta\in\R$, the map $R_\theta z=e^{i\theta}z$ is an orthogonal transformation of $\R^{2N_j}$. The orthogonal equivariance in \Cref{lem:P2-equiv} therefore gives $q_j\circ R_\theta=q_j$ almost everywhere with respect to $\gamma^{2N_j}$. Both sides are continuous quadratic polynomials, and the support of the Gaussian measure is all of $\R^{2N_j}$. Thus this equality holds at every point, and $q_j$ is $U(1)$-invariant. By \Cref{lem:U1-H2},
\begin{equation*}
 \abs{\mathbb Eq_j^3}\le2(\mathbb Eq_j^2)^{3/2}.
\end{equation*}
Passing to the limit and using \eqref{eq:qj-mom-limit}, we obtain
\[
 8\le2\cdot2^{3/2}=4\sqrt2.
\]
This contradicts $4\sqrt2<8$. Therefore, $H_1\notin\Gamma^\infty/U(1)$. Since \Cref{lem:H-membership} gives $H_1\in\Gamma^\infty/\Z_2$, the inclusion \eqref{eq:natural-inclusion} is strict.
\end{proof}

\begin{proof}[Proof of \Cref{thm:intro-strict-containment}]
The theorem follows from \Cref{eq:natural-inclusion} and
\Cref{thm:H1-not-U1}.
\end{proof}

\subsection{The two quotients do not coincide even at different scales}

\begin{corollary} \label{cor:quot-scale}
For every $b,c>0$,
\[
 b(\Gamma^\infty/\Z_2)\ne c(\Gamma^\infty/U(1)).
\]
\end{corollary}

\begin{proof}
If the two pyramids were equal, \Cref{thm:C22} and $1$-homogeneity would give
\[
 \frac b{\sqrt2}
 =C_{2,2}(b(\Gamma^\infty/\Z_2))
 =C_{2,2}(c(\Gamma^\infty/U(1)))
 =\frac c{\sqrt2}.
\]
Thus $b=c$. Rescaling both sides by $1/b$ would give $\Gamma^\infty/\Z_2=\Gamma^\infty/U(1)$, contradicting \Cref{thm:H1-not-U1}.
\end{proof}

\begin{proof}[Proof of \Cref{thm:main}]
By \Cref{cor:Gauss-vs-quot}, the Gaussian pyramid is non-similar to either Hopf quotient. By \Cref{cor:quot-scale}, the two Hopf quotients are non-similar. Thus the three pyramids are pairwise non-similar.
\end{proof}

\begin{ack}
The authors would like to thank Tomohiro Fukaya, Ayato Mitsuishi and Daisuke Kazukawa for their valuable comments.
The authors used ChatGPT in preparing this manuscript. The authors reviewed and revised the mathematical content and take full responsibility for the final manuscript.
\end{ack}

\end{document}